\documentclass[a4paper,11pt,reqno]{amsart}

\usepackage[margin=1in,includehead,includefoot]{geometry}
\usepackage[utf8]{inputenc}
\usepackage[T1]{fontenc}
\usepackage{lmodern,microtype}
\usepackage{upref}
\usepackage{mathtools,amssymb,amsthm,amsmath}
\usepackage[foot]{amsaddr}
\usepackage{graphicx}
\usepackage{hyperref}
\usepackage{wrapfig}
\usepackage[dvipsnames]{xcolor}
\usepackage{enumitem}
\usepackage[margin=1cm]{caption}
\usepackage[textsize=footnotesize,disable]{todonotes}
\usepackage{bookmark}
\usepackage{mycommands}

\graphicspath{{./graphics/}}

\newcommand{\torsten}[1]{\todo[inline,color=yellow]{\textbf{Torsten:} #1}}

\newtheorem{theorem}{Theorem}
\newtheorem{lemma}[theorem]{Lemma}
\newtheorem{conjecture}[theorem]{Conjecture}

\makeatletter
\let\old@setaddresses\@setaddresses
\def\@setaddresses{\bigskip{\parindent 0pt\let\scshape\relax\let\ttfamily\relax\old@setaddresses}}
\makeatother


\hypersetup{
  pdftitle={Flips in colorful triangulations},
  pdfauthor={Rohan Acharya, Torsten M\"utze, and Francesco Verciani}
}

\title{Flips in colorful triangulations}

\author{Rohan Acharya}
\address[Rohan Acharya]{Department of Computer Science, University of Warwick, United Kingdom}
\email{rohan.acharya@warwick.ac.uk}

\author{Torsten M\"utze}
\address[Torsten M\"utze]{Institut f\"ur Mathematik, Universit\"at Kassel, Kassel, Germany \& Department of Theoretical Computer Science and Mathematical Logic, Charles University, Prague, Czech Republic}
\email{tmuetze@mathematik.uni-kassel.de}

\author{Francesco Verciani}
\address[Francesco Verciani]{Institut f\"ur Mathematik, Universit\"at Kassel, Kassel, Germany}
\email{francesco.verciani@uni-kassel.de}

\thanks{An extended abstract of this work was accepted for presentation at the 32nd International Symposium on Graph Drawing and Network Visualization (Graph Drawing 2024).}

\thanks{This work was supported by Czech Science Foundation grant GA~22-15272S. The second and third author participated in the workshop `Combinatorics, Algorithms and Geometry' in March 2024, which was funded by German Science Foundation grant~522790373.}

\begin{document}
\begin{abstract}
The associahedron is the graph~$\cG_N$ that has as nodes all triangulations of a convex $N$-gon, and an edge between any two triangulations that differ in a flip operation.
A \emph{flip} removes an edge shared by two triangles and replaces it by the other diagonal of the resulting 4-gon.
In this paper, we consider a large collection of induced subgraphs of~$\cG_N$ obtained by Ramsey-type colorability properties.
Specifically, coloring the points of the $N$-gon red and blue alternatingly, we consider only \emph{colorful} triangulations, namely triangulations in which every triangle has points in both colors, i.e., monochromatic triangles are forbidden.
The resulting induced subgraph of~$\cG_N$ on colorful triangulations is denoted by~$\cF_N$.
We prove that~$\cF_N$ has a Hamilton cycle for all~$N\geq 8$, resolving a problem raised by Sagan, i.e., all colorful triangulations on $N$ points can be listed so that any two cyclically consecutive triangulations differ in a flip.
In fact, we prove that for an arbitrary fixed coloring pattern of the $N$ points with at least 10 changes of color, the resulting subgraph of~$\cG_N$ on colorful triangulations (for that coloring pattern) admits a Hamilton cycle.
We also provide an efficient algorithm for computing a Hamilton path in~$\cF_N$ that runs in time~$\cO(1)$ on average per generated node.
This algorithm is based on a new and algorithmic construction of a tree rotation Gray code for listing all $n$-vertex $k$-ary trees that runs in time~$\cO(k)$ on average per generated tree.
\end{abstract}

\maketitle

\section{Introduction}

The \defi{associahedron} is a polytope of fundamental interest and importance~\cite{MR2108555,MR2321739,MR3437894}, as it lies at the heart of many recent developments in algebraic combinatorics and discrete geometry; see~\cite{MR4675114} and the references therein.
In this paper we are specifically interested in its combinatorial structure, namely the graph of its skeleton; see Figure~\ref{fig:asso}.
This graph, which we denote by~$\cG_N$, has as nodes all triangulations of a convex $N$-gon ($N\geq 3$), and an edge between any two triangulations that differ in a \defi{flip} operation, which consists of removing an edge shared by two triangles and replacing it by the other diagonal of the resulting 4-gon.
The graph~$\cG_N$ is isomorphic to the graph that has as nodes all binary trees with $N-2$ vertices, and an edge between any two trees that differ in a tree rotation.
Each binary tree arises as the geometric dual of a triangulation, i.e., we place a vertex into every triangle, connect vertices in adjacent triangles by an edge, and the root of the tree is given by `looking through' a fixed outer edge; see Figure~\ref{fig:asso}.
Under this bijection between triangulations and binary trees, flips translate to tree rotations; see Figure~\ref{fig:flip2}.
This dual point of view between triangulations and binary trees, and more generally, between dissections into $(k+1)$-gons and $k$-ary trees, is an essential tool used in our paper.

\begin{figure}[t!]
\includegraphics[page=1]{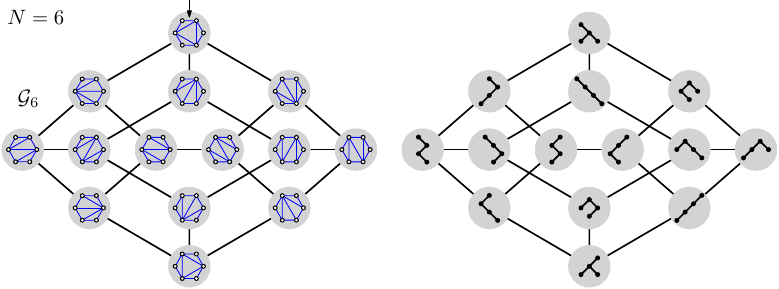}
\caption{The graph of the 3-dimensional associahedron.
The top edge of each triangulation is the outer edge that determines the root of the corresponding binary tree (see little arrow).}
\label{fig:asso}
\end{figure}

Properties of the graph~$\cG_N$ have been the subject of extensive investigations in the literature.
Most prominently, the diameter of~$\cG_N$ was shown to be~$2N-10$ for all $N>12$~\cite{MR928904,MR3197650}.
Furthermore, the graph~$\cG_N$ is regular with degree~$N-3$, and this number is also its connectivity~\cite{MR1723053}.
The chromatic number of~$\cG_N$ is at most~$\cO(\log N)$~\cite{MR2535071,berry-et-al:18}, while the best known lower bound is only~4.

\begin{wrapfigure}{r}{0.45\textwidth}
\centering
\includegraphics[page=2]{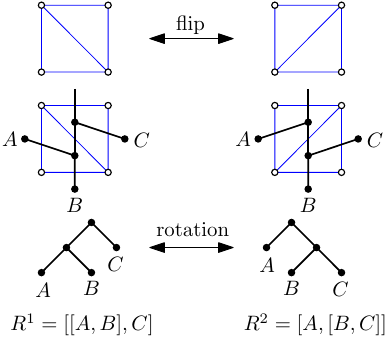}
\caption{Correspondence between flips in triangulations (top) and rotations in binary trees (bottom).}
\label{fig:flip2}
\end{wrapfigure}
Another fundamental graph property that we focus on in this paper is Hamiltonicity.
To this end, Lucas~\cite{MR920505} first proved that $\cG_N$ admits a Hamilton cycle for~$N\geq 5$, and a short proof was given by Hurtado and Noy~\cite{MR1723053}.
A Hamilton path in~$\cG_N$ can be computed efficiently and yields a Gray code ordering of all binary trees by rotations~\cite{MR1239499}.
This algorithm is a special case of the more general Hartung-Hoang-M\"utze-Williams permutation language framework~\cite{MR4391718,MR4344032,MR4598046,DBLP:journals/talg/CardinalMM25}.

In this paper, we consider a large collection of induced subgraphs of~$\cG_N$ obtained by Ramsey-type colorability properties.
This line of inquiry was initiated by Sagan~\cite{MR2426410}, following a sequence of problems posed by Propp on a mailing list in~2003.
Specifically, we label the points of the convex $N$-gon by $1,\ldots,N$ in counterclockwise order, and we color them red (\tr) and blue (\tb) alternatingly.
It follows that point~$i$ is colored red if $i$ is odd and blue if $i$ is even.
Note that for even~$N$, any two neighboring points have opposite colors, whereas for odd~$N$ this property is violated for the first and last point, which are both red.

\begin{figure}[h!]
\includegraphics[page=1]{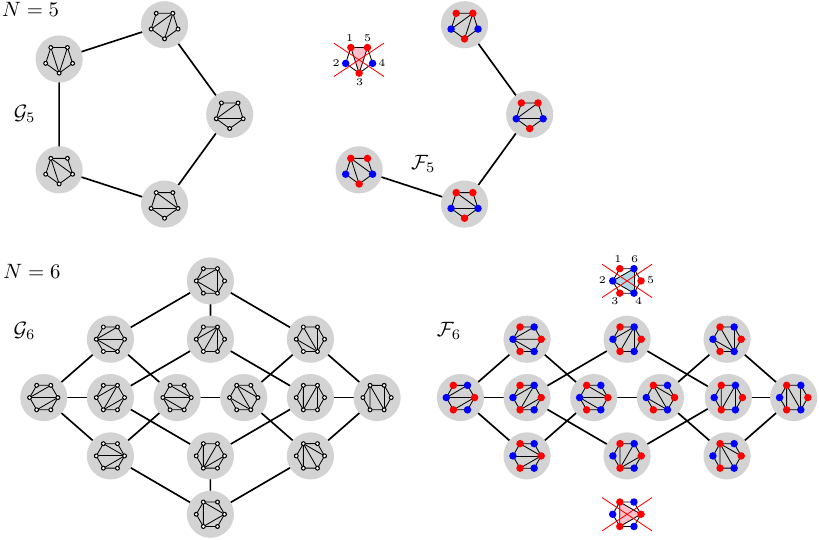}
\caption{Induced subgraphs of the associahedron~$\cG_5$ (top left) and~$\cG_6$ (bottom left) obtained for the coloring sequence $\tr\tb\tr\tb\cdots$ by forbidding monochromatic triangles.
The triangulations with monochromatic triangles are still shown, but they are not part of the graphs~$\cF_5$ and~$\cF_6$ (top right and bottom right, respectively) and hence crossed out.
}
\label{fig:f56}
\end{figure}

We say that a triangulation is \defi{colorful} if every triangle has points of both colors, i.e., no triangles in which all three points have the same color.
We write~$\cF_N$ for the subgraph of~$\cG_N$ induced by all colorful triangulations.
In other words, $\cF_N$ is obtained from~$\cG_N$ by deleting all triangulations that have a monochromatic triangle; see Figure~\ref{fig:f56}.

\subsection{Sagan's problem and its generalization}

Sagan~\cite{MR2426410} proved that~$\cF_N$ is a connected graph, and he asked [personal communication] whether~$\cF_N$ admits a Hamilton path or cycle.
Looking at the first two interesting instances~$N=5$ and~$N=6$ in Figure~\ref{fig:f56}, we note that~$\cF_5$ has a Hamilton path, but no cycle, and~$\cF_6$ has no Hamilton path and hence no cycle either.
Furthermore, $\cF_7$ admits a Hamilton path (see Figure~\ref{fig:F4567}), but no Hamilton cycle, which seems rather curious (cf.\ Theorem~\ref{thm:a11111} below).
We prove the following result.

\begin{theorem}
\label{thm:FN-ham}
For any $N\geq 8$, the graph~$\cF_N$ has a Hamilton cycle.
\end{theorem}

The resolution of Sagan's question immediately gives rise to the following more general problem:
We consider an arbitrary sequence~$\alpha$ of coloring the points~$1,\ldots,N$ red or blue, and let~$\cF_\alpha$ be the corresponding induced subgraph of~$\cG_N$ obtained by forbidding monochromatic triangles.
For which sequences~$\alpha$ does~$\cF_\alpha$ admit a Hamilton path or cycle?

Formally, a \defi{coloring sequence} is a sequence $\alpha=(\alpha_1,\ldots,\alpha_\ell)$ of even length $\ell\geq 2$ with $\alpha_i\geq 1$ for $i=1,\ldots,\ell$, and it encodes the coloring pattern
\begin{equation}
\label{eq:col-pat}
\tr^{\alpha_1}\tb^{\alpha_2}\tr^{\alpha_3}\tb^{\alpha_4}\cdots \tr^{\alpha_{\ell-1}}\tb^{\alpha_\ell}
\end{equation}
for the points~$1,\ldots,N$, where $N=\sum_{i=1}^\ell \alpha_i$, and $\tr^{\alpha_i}$ and $\tb^{\alpha_j}$ denote $\alpha_i$-fold and $\alpha_j$-fold repetition of red and blue, respectively.
In words, the first~$\alpha_1$ many points are colored red, the next~$\alpha_2$ many points are colored blue, the next $\alpha_3$ many points are colored red etc.
Clearly, the special cases considered by Sagan are $\alpha_1=\alpha_2=\cdots=\alpha_\ell=1$ for even $N=\ell$, or $\alpha_1=2$ and $\alpha_2=\cdots=\alpha_\ell=1$ for odd $N=\ell+1$, respectively (in the second case, the two consecutive points of the same color are~$1$ and~$2$ instead of~1 and~$N$ as before, but this is only a cyclic shift of indices).
We let $\cF_\alpha$ be the induced subgraph of~$\cG_N$ induced by the colorful triangulations with coloring sequence~$\alpha$.

We provide the following generalization of Theorem~\ref{thm:FN-ham} before.
Specifically, our next theorem applies to all coloring patterns with at least 10 changes of colors.

\begin{theorem}
\label{thm:Falpha-ham}
For any coloring sequence $\alpha=(\alpha_1,\ldots,\alpha_\ell)$ of (even) length~$\ell\geq 10$, the graph~$\cF_\alpha$ has a Hamilton cycle.
\end{theorem}

Note that there are~$2^{N-2}$ different coloring sequences satisfying the conditions of the theorem, i.e., there are exponentially many subgraphs of the associahedron to which Theorem~\ref{thm:Falpha-ham} applies.
This also shows that the associahedron has cycles of many different lengths.

In view of the last theorem, it remains to consider short coloring sequences, i.e., sequences of length~$\ell\leq 8$.
We offer three simple observations in this regime.
We first consider the easiest case~$\ell=2$, i.e., the coloring sequence has the form~$\alpha=(a,b)$.
The resulting graph~$\cF_\alpha$ for $\alpha=(4,4)$ is shown in Figure~\ref{fig:comb}.
Another way to think about such a triangulation is as a triangulation of the so-called double-chain, where each triangle has to touch both chains.
We observe that the number of colorful triangulations in this case is $\binom{N-2}{a-1}=\binom{N-2}{b-1}$ where $N:=a+b$.
Moreover, these triangulations are in bijection with bitstrings of length~$N-2$ with $a-1$ many 0s and $b-1$ many 1s, so-called \defi{$(a-1,b-1)$-combinations}.
This bijection is defined as follows; see Figure~\ref{fig:comb}:
Given a triangulation, we consider a ray separating the red from the blue points, and we record the types of triangles intersected by this ray one after the other, specifically we record a 1-bit or 0-bit if the majority color of the three triangle points is red and blue, respectively.
We see that flips in the triangulations correspond to adjacent transpositions in the corresponding bitstrings.
In the following, we use the generic term \defi{flip graph} for any graph that has as nodes a set of combinatorial objects, and an edge between any two objects that differ in a certain change operation.
From what we said before, it follows that $\cF_{(a,b)}$ is isomorphic to the flip graph of $(a-1,b-1)$-combinations under adjacent transpositions.
We can thus apply known results from~\cite{MR737262,MR821383} to obtain the following theorem.

\begin{figure}[h!]
\includegraphics[page=8]{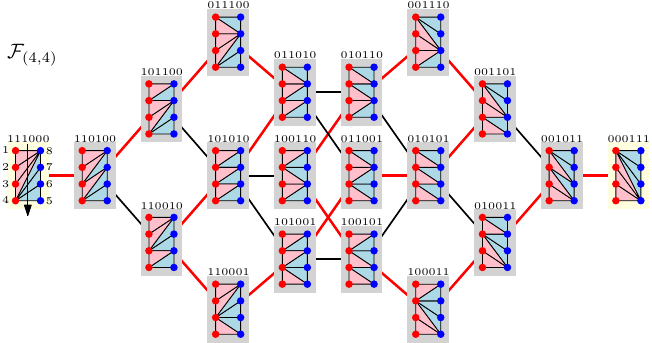}
\caption{Flip graph of colorful triangulations for the coloring sequence $\alpha=(4,4)$ (\tr\tr\tr\tr\tb\tb\tb\tb), which is isomorphic to the flip graph of $(3,3)$-combinations under adjacent transpositions.
The black arrow in the leftmost triangulation is the ray that separates red from blue points, and the combination is obtained by reading the triangle types that intersect this ray from top to bottom (red=1, blue=0).
The nodes of degree~1 and a Hamilton path in the flip graph are highlighted.
}
\label{fig:comb}
\end{figure}

\begin{theorem}
\label{thm:comb}
For integers~$a,b\geq 1$ with $a+b\geq 3$, the graph~$\cF_{(a,b)}$ is isomorphic to the flip graph of $(a-1,b-1)$-combinations under adjacent transpositions.
Consequently, $\cF_{(a,b)}$ has a Hamilton path if and only if $a\in\{1,2\}$, or $b\in\{1,2\}$, or $a$ and~$b$ are both even.
Furthermore, if $a,b\geq 2$, then $\cF_{(a,b)}$ has no Hamilton cycle.
\end{theorem}

The reason for the non-existence of a Hamilton cycle is that~$\cF_{(a,b)}$ has two nodes of degree~1, corresponding to the combinations~$1^{a-1}0^{b-1}$ and~$0^{b-1}1^{a-1}$; see Figure~\ref{fig:comb}.

The next result is a simple observation for the special case of coloring sequences of length~$\ell=4$ with exactly two non-consecutive blue points; see Figure~\ref{fig:grid}.

\begin{theorem}
\label{thm:grid}
For integers $a,b\geq 1$, the graph~$\cF_{(a,1,b,1)}$ is isomorphic to an $a\times b$ rectangular grid with one pending edge attached to each node.
Consequently, it does not have a Hamilton path unless $a\cdot b\leq 2$.
\end{theorem}

The nodes of degree~1 are the triangulations in which the two blue points are not connected by an edge, in which case the only possible flip restores this edge between the two blue points.

\begin{figure}[h!]
\includegraphics[page=9]{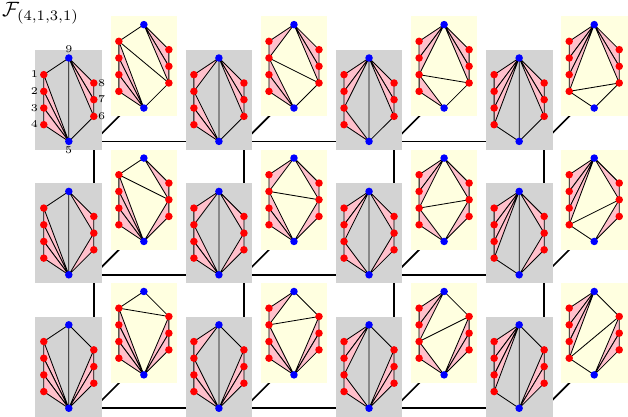}
\caption{Illustration of Theorem~\ref{thm:grid} for the coloring sequence~$\alpha=(4,1,3,1)$ ($\tr\tr\tr\tr\tb\tr\tr\tr\tb$).
The nodes of degree~1 in the flip graph are highlighted.}
\label{fig:grid}
\end{figure}

The last result is for coloring sequences of length $\ell=6$ and yields an infinite family of natural flip graphs that admit a Hamilton path but no Hamilton cycle, despite the fact that they have minimum degree~2; see Figure~\ref{fig:a11111}.

\begin{theorem}
\label{thm:a11111}
For $\alpha=(a,1,1,1,1,1)$, the graph~$\cF_\alpha$ has no Hamilton cycle for any~$a\geq 1$, but it has a Hamilton path for~$a=2$ and any~$a\geq 4$.
\end{theorem}

\subsection{Algorithmic questions and higher arity}

We also provide an algorithmic version of Theorem~\ref{thm:FN-ham}.

\begin{theorem}
\label{thm:FN-algo}
For any $N\geq 8$, a Hamilton path in the graph~$\cF_N$ can be computed in time~$\cO(1)$ on average per node.
\end{theorem}

The initialization time and memory requirement for this algorithm are~$\cO(N)$.
We later discuss the data structures used to represent colorful triangulations in our program.

\begin{figure}[t!]
\includegraphics[page=1]{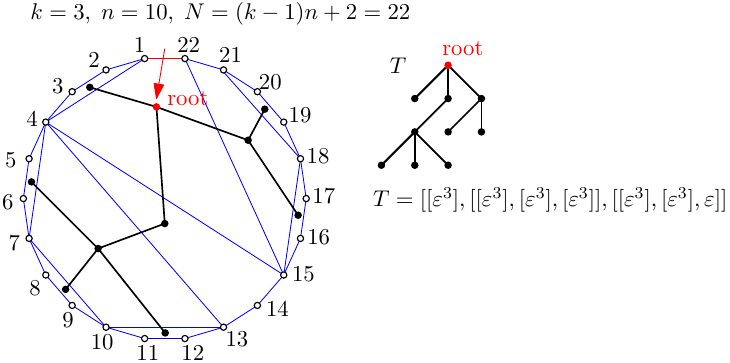}
\caption{Bijection between dissections of an $N$-gon into $(k+1)$-gons and $k$-ary trees, illustrated for the case~$k=3$.}
\label{fig:diss}
\end{figure}

\begin{figure}[t!]
\includegraphics[page=3]{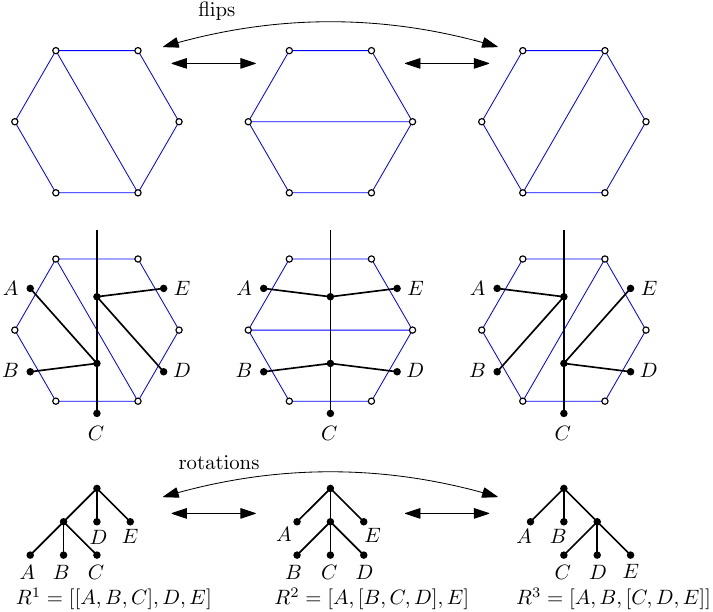}
\caption{Correspondence between flips in quadrangulations (top) and rotations in ternary trees (bottom); cf. Figure~\ref{fig:flip2}.}
\label{fig:flip3}
\end{figure}

Our construction of a Hamilton path/cycle in~$\cF_N$ relies on a Gray code ordering of ternary trees by rotations.
We first describe this setup, generalizing our earlier definitions about triangulations and binary trees; see Figures~\ref{fig:diss} and~\ref{fig:flip3} for illustration.
Let $k\geq 2$ and $n\geq 1$ be integers, and let $N:=(k-1)n+2$.
We consider a \defi{dissection} of a convex $N$-gon into $n$ many $(k+1)$-gons.
A \defi{flip} operation removes an edge shared by two $(k+1)$-gons and replaces it by one of the other $k-1$ possible diagonals of the resulting $2k$-gon.
Dissections of an $N$-gon into $(k+1)$-gons are in bijection with $k$-ary trees with $n$ vertices.
Each $k$-ary trees arises as the geometric dual of a dissection into $(k+1)$-gons, with the root given by `looking through' the outer edge $1N$, and flips translate to tree rotations under this bijection.

We denote the corresponding flip graph of dissections of an $N$-gon into $(k+1)$-gons by $\cG_{N,k+1}$.
The associahedron is the special case $k=2$, i.e., the graph~$\cG_{N,3}=\cG_N$.
By what we said before, the graph~$\cG_{N,k+1}$ is isomorphic to the rotation graph of $k$-ary trees with $n$ vertices, where $N=(k-1)n+2$.
Huemer, Hurtado, and Pfeifle~\cite{MR2474724} first proved that $\cG_{N,k+1}$ has a Hamilton cycle for all $k\geq 3$, which combined with the results of Hurtado and Noy~\cite{MR1723053} for the case $k=2$ (binary trees) yields the following theorem.

\begin{theorem}[\cite{MR1723053} for $k=2$; \cite{MR2474724} for $k\geq 3$]
\label{thm:GNk-ham}
For any $k\geq 2$, $n\geq \max\{2,5-k\}$ and $N:=(k-1)n+2$, the graph~$\cG_{N,k+1}$ has a Hamilton cycle.
\end{theorem}

The proof from~\cite{MR2474724} for the case $k\geq 3$ does not generalize the simple inductive construction of a Hamilton path/cycle in the associahedron (the case $k=2$) described in~\cite{MR1723053}, and it imposes substantial difficulties when translating it to an efficient algorithm.
Consequently, we provide a unified and simplified proof for Theorem~\ref{thm:GNk-ham}, valid for all $k\geq 2$, which can be turned into an efficient algorithm.
This result generalizes the efficient algorithm for computing a Hamilton path in the associahedron provided by Lucas, Roelants van Baronaigien, and Ruskey~\cite{MR1239499}.

\begin{theorem}
\label{thm:GNk-algo}
For any $k\geq 2$, $n\geq \max\{2,5-k\}$ and $N:=(k-1)n+2$, a Hamilton path in~$\cG_{N,k+1}$ can be computed in time~$\cO(k)$ on average per node.
\end{theorem}

\begin{figure}[b!]
\centering
\begin{minipage}{.5\textwidth}
\centering
\includegraphics[page=1]{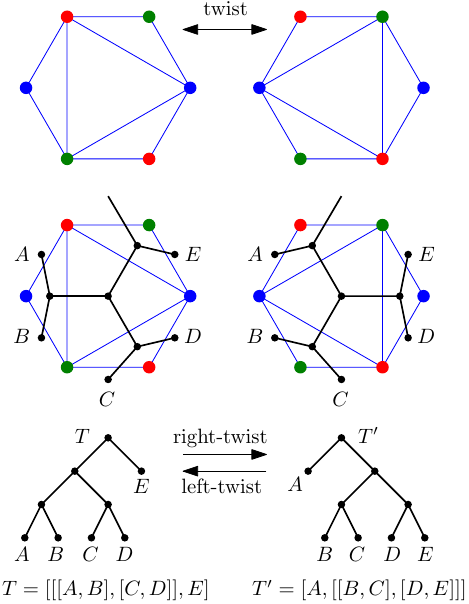}
\captionof{figure}{Twist operation and the corresponding binary trees.}
\label{fig:twist}
\end{minipage}%
\begin{minipage}{.5\textwidth}
\centering
\includegraphics{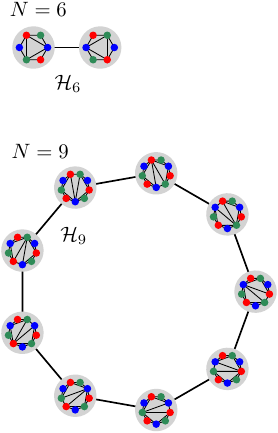}
\captionof{figure}{The flip graphs~$\cH_6$ and~$\cH_9$.}
\label{fig:H69}
\end{minipage}
\end{figure}

The initialization time and memory requirement for this algorithm are~$\cO(kn)$.

We implemented the algorithms mentioned in Theorems~\ref{thm:FN-algo} and~\ref{thm:GNk-algo} in C++, and made the code available for download, experimentation and visualization on the Combinatorial Object Server~\cite{cos_kary}.
In fact, we also provide a slightly more technical implementation for efficiently computing a Hamilton cycle in~$\cG_{N,k+1}$ for arbitrary~$k\geq 2$.  

\subsection{Three colors}

We now consider colorings of the points~$1,\ldots,N$ with more than two colors.
To start with, we color the points in counterclockwise order alternatingly red (\tr), blue (\tb) and green (\tg), and we consider triangulations in which every triangle has points of all three colors, i.e., one point of each color.
This setting has also been considered by Sagan~\cite{MR2426410}.
Note that flips of a single diagonal as before are not valid operations anymore (in the sense that the flip graph would not have any edges), so we consider a modified flip operation instead which consists of a particular sequence of 4~flips.
Specifically, a \defi{twist} `rotates' a triangle that is surrounded by three triangles, i.e., the inner triangle is removed, creating an empty 6-gon, and the triangle is inserted the other way; see Figure~\ref{fig:twist}.
We write $\cH_N$ for the flip graph of colorful triangulations under twists; see Figure~\ref{fig:H69}.

\begin{theorem}
\label{thm:HN-conn}
For any $N$ that is a multiple of~3, the graph~$\cH_N$ is connected.
\end{theorem}

\subsection{Outline of this paper}

In Section~\ref{sec:prelim} we collect a few definitions and auxiliary results used throughout this paper.
In Section~\ref{sec:colorful} we prove Theorems~\ref{thm:FN-ham} and~\ref{thm:Falpha-ham}.
In Section~\ref{sec:a11111} we prove Theorem~\ref{thm:a11111}.
In Section~\ref{sec:unified} we present a new proof of Theorem~\ref{thm:GNk-ham}, which forms the basis of the algorithmic results.
In Section~\ref{sec:algo} we we prove Theorems~\ref{thm:FN-algo} and~\ref{thm:GNk-algo}.
In Section~\ref{sec:3col} we turn to the setting with three colors, proving Theorem~\ref{thm:HN-conn}.
We conclude with some open questions in Section~\ref{sec:open}.

\section{Preliminaries}
\label{sec:prelim}

In this section we introduce some notations that will be used throughout this paper.
For the reader's convenience, some frequently used symbols are summarized in Table~\ref{tab:notation}.

\begin{table}
\caption{Summary of frequently used notations.}
\label{tab:notation}
\begin{tabular}{ll}
\bf Symbol & \bf Description \\ \hline
$N$ & number of points \\
$k$ & arity of the trees, dissection into $(k+1)$-gons \\
$\cG_N$ & associahedron, i.e., flip graph on all triangulations \\
$\cT_{n,k}$ & set of $k$-ary trees with $n$ vertices \\
$\cD_{N,k+1}$ & set of dissections of $N$-gon into $(k+1)$-gons \\
$\cG_{N,k+1}$ & flip graph on dissections~$\cD_{N,k+1}$ \\ \hline
$\cC_N$ & set of colorful triangulations for alternating red-blue coloring \\
$\cF_N\seq \cG_N$ & flip graph on $\cC_N$ \\
$r(T)$ & reduced dissection obtained from triangulation $T$ \\
$\cF_N'$ & flip graph on reduced dissections \\ \hline
$\alpha=(\alpha_1,\ldots,\alpha_\ell)$ & coloring sequence for two colors, $N=\sum_{i=1}^\ell \alpha_i$ \\
$\cC_\alpha$ & set of colorful triangulations for coloring sequence~$\alpha$ \\
$\cF_\alpha$ & flip graph on $\cC_\alpha$ \\
$E_\alpha$ & monochromatic boundary edges \\
$q$ & number of quadrangles \\
$t$ & number of triangles \\
$\cD_\alpha$ & set of reduced dissections into quadrangles and triangles \\
$\cF_\alpha'$ & flip graph on $\cD_\alpha$ \\ \hline
$\cC_N'$ & set of colorful triangulations for alternating red-blue-green coloring \\
$\cH_N$ & flip graph on $\cC_N'$ \\ \hline
$\varepsilon$ & empty string \\
$\rev(x)$ & reverse of sequence~$x$ \\
$f(x)$, $\ell(x)$ & first and last entry of sequence~$x$ \\
$v(G)$, $\Delta(G)$ & number of vertices of graph~$G$, maximum degree of~$G$ \\
$G\simeq H$ & isomorphic graphs $G$ and~$H$ \\
$Q_d$ & $d$-dimensional hypercube \\ \hline
$\nu_T(S)$ & subtree sequence for subtree~$S$ of~$T$ \\
$\rho(T)$ & rightmost branch of~$T$ \\
$T\sim T'$ & two trees that differ in a rotation \\
$c(T)$ & children sequence of tree~$T$ \\
\end{tabular}
\end{table}

\subsection{String operations}

We write~$\varepsilon$ for the empty string.
For any string~$x$ and any integer~$k\geq 0$, we write $x^k$ for the $k$-fold concatenation of~$x$.
Given any sequence $x=(x_1,\ldots,x_\ell)$, we write $\rev(x):=(x_\ell,x_{\ell-1},\ldots,x_1)$ for the reversed sequence.

\subsection{Dissections and trees}

For integers $k\geq 2$ and $n\geq 1$, let $N:=(k-1)n+2$.
We write~$\cD_{N,k+1}$ for the set of all dissections of a convex $N$-gon into~$(k+1)$-gons.
In particular, $\cD_{N,3}$ are triangulations of a convex $N$-gon.
We write $\cT_{n,k}$ for the set of all $k$-ary trees with $n$ vertices, and $t_{n,k}:=|\cT_{n,k}|$.
Both objects are counted by the \defi{$k$-Catalan numbers} (OEIS sequence~A062993), i.e., we have
\begin{equation*}
|\cD_{N,k+1}|=|\cT_{n,k}|=t_{n,k}=\frac{1}{(k-1)n+1}\binom{kn}{n};
\end{equation*}
see Table~\ref{tab:kCat} .
We also define $t_{n,3}':=\sum_{i=0}^n t_{i,3}\cdot t_{n-i,3}$ as the number of pairs of ternary trees with $n$ vertices in total.
We have the explicit formula
\begin{equation*}
t_{n,3}'=\frac{1}{n+1}\binom{3n+1}{n}
\end{equation*}
(OEIS~A006013).

\begin{table}[t!]
\caption{Counts $|\cD_{N,k+1}|=|\cT_{n,k}|$ of dissections and trees, for $k=2,\ldots,10$ and $n=1,\ldots,10$ ($N=(k-1)n+2$).}
\label{tab:kCat}
\begin{tabular}{l|llllllllll}
 & \multicolumn{10}{|c}{$n$} \\
$k$ & 1 & 2 & 3 & 4 & 5 & 6 & 7 & 8 & 9 & 10 \\ \hline
2   & 1 & 2 & 5 & 14 & 42 & 132 & 429 & 1430 & 4862 & 16796 \\
3   & 1 & 3 & 12 & 55 & 273 & 1428 & 7752 & 43263 & 246675 & 1430715 \\
4   & 1 & 4 & 22 & 140 & 969 & 7084 & 53820 & 420732 & 3362260 & 27343888 \\
5   & 1 & 5 & 35 & 285 & 2530 & 23751 & 231880 & 2330445 & 23950355 & 250543370 \\
6   & 1 & 6 & 51 & 506 & 5481 & 62832 & 749398 & 9203634 & 115607310 & 1478314266 \\
7   & 1 & 7 & 70 & 819 & 10472 & 141778 & 1997688 & 28989675 & 430321633 & 6503352856 \\
8   & 1 & 8 & 92 & 1240 & 18278 & 285384 & 4638348 & 77652024 & 1329890705 & 23190029720 \\
9   & 1 & 9 & 117 & 1785 & 29799 & 527085 & 9706503 & 184138713 & 3573805950 & 70625252863 \\
10  & 1 & 10 & 145 & 2470 & 46060 & 910252 & 18730855 & 397089550 & 8612835715 & 190223180840 \\
\end{tabular}
\end{table}

For any coloring sequence~$\alpha$, we write~$\cC_\alpha$ for the set of colorful triangulations with coloring pattern defined in~\eqref{eq:col-pat}.
By these definitions, $\cF_\alpha$ is the subgraph of~$\cG_N$ induced by the triangulations in~$\cC_\alpha$.
Sagan's question concerned the special case $\alpha:=1^N$ for even~$N$ and $\alpha:=(2,1^{N-2})$ for odd~$N$, and for those particular coloring sequences~$\alpha$ we simply write~$\cC_N=\cC_\alpha$ and~$\cF_N=\cF_\alpha$.
Sagan proved the following.

\begin{theorem}[{\cite[Thm.~2.1]{MR2426410}}]
\label{thm:count}
For any $q\geq 1$ we have
\[
|\cC_N|=\begin{cases}
2^q\cdot t_{q,3}=\frac{2^q}{2q+1}\binom{3q}{q} & \text{if } N=2q+2, \\
2^q\cdot t_{q,3}'=\frac{2^q}{q+1}\binom{3q+1}{q} & \text{if } N=2q+3.
\end{cases}
\]
\end{theorem}

The two sequences in this theorem are OEIS~A153231 and~A369510, respectively.

\subsection{Three colors}
\label{sec:prelim-3col}

We write~$\cC_N'$ for the set of triangulations for which the points $1,\ldots,N$ are colored red ($\tr$), blue ($\tb$), green ($\tg$) alternatingly, and where every triangle has points of all three colors.
In such a triangulation~$T\in\cC_N'$, there are two types of triangles, those in which a counterclockwise reading of its colors gives $\tr\tb\tg$ and those which give $\tr\tg\tb$, and we refer to them as type~0 and type~1, respectively; see Figure~\ref{fig:bw}.
Clearly, any two triangles that share an edge have opposite types.
Furthermore, if $N$ is a multiple of~3, then all triangles on the boundary are of type~0.
Equivalently, every type~1 triangle is not on the boundary but is surrounded by three type~0 triangles.
It follows that in the corresponding binary tree, all vertices in odd distance from the root have two children, i.e., we have a bijection between~$\cC_N'$ and binary trees with $N-2$ vertices in which all vertices in odd distance from the root have two children.
By taking shortcuts between consecutive odd distance vertices, we obtain a bijection with pairs of 4-ary trees with in total $p:=(N-3)/3$ vertices.
These families of objects are counted by~$|\cC_N'|=\frac{1}{3p+2}\binom{4p+1}{p}$ (OEIS A069271); cf.~\cite[Thm.~2.2]{MR2426410}.
Similar formulas can be derived when $N$ is not a multiple of~3; see~\cite{MR2426410}.

\begin{figure}[h!]
\includegraphics{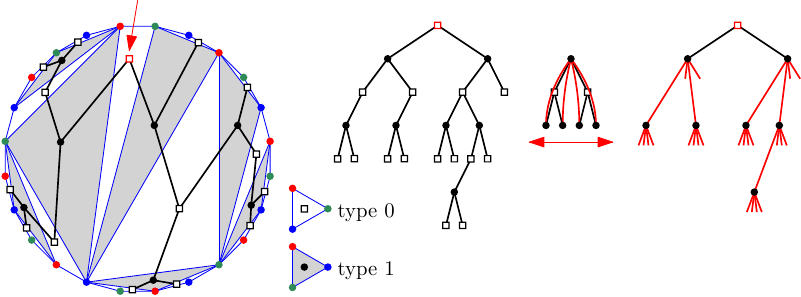}
\caption{Triangle types in 3-colored triangulations and bijection with degree-constrained binary trees and pairs of 4-ary trees.}
\label{fig:bw}
\end{figure}

\subsection{Graphs}
\label{sec:graphs}

For a graph~$G$, we write~$v(G)$ for the number of vertices of~$G$, and~$\Delta(G)$ for its maximum degree.
Also, we write $G\simeq H$ for two graphs~$G$ and~$H$ that are isomorphic.

For any integer~$d\geq 1$, the \defi{$d$-dimensional hypercube~$Q_d$} is the graph that has as vertices all bitstrings of length~$d$, and an edge between any two strings that differ in a single bit.

\begin{lemma}[\cite{MR2178189}]
\label{lem:edge-ham}
For any $d\geq 2$ and any set~$E$ of at most $2d-3$ edges in~$Q_d$ that together form vertex-disjoint paths, there is a Hamilton cycle that contains all edges of~$E$.
\end{lemma}

For integers~$a\geq 1$ and~$d\geq 1$ we define $S(a,d)$ as the set of all $a$-tuples of non-decreasing integers from the set~$\{1,\ldots,d\}$, i.e., $S(a,d)=\{(j_1,\ldots,j_a)\mid 1\leq j_1\leq j_2\leq \cdots\leq j_a\leq d\}$.
Furthermore, we let $G(a,d)$ be the graph with vertex set~$S(a,d)$ and edges between any two $a$-tuples that differ in a single entry by~$\pm 1$; see Figure~\ref{fig:Gad}.

\begin{figure}[h!]
\includegraphics{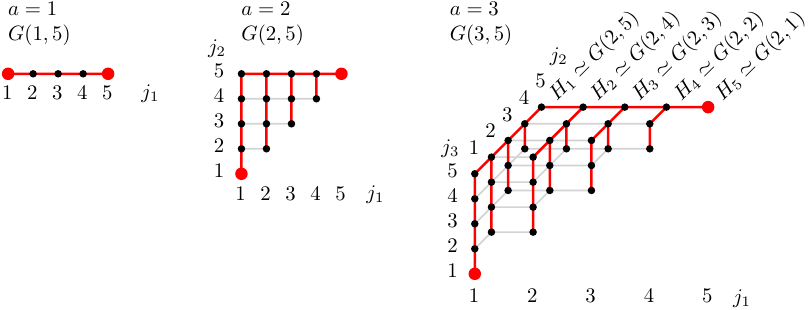}
\caption{Illustration of the graph~$G(a,d)$ and Lemma~\ref{lem:Gad}.
The spanning trees and the extremal vertices are highlighted.}
\label{fig:Gad}
\end{figure}

\begin{lemma}
\label{lem:Gad}
For any $a\geq 1$ and~$d\geq 1$, the graph~$G(a,d)$ has a spanning tree~$T$ with $\Delta(T)\leq 3$.
\end{lemma}

We refer to the vertices~$1^a$ and~$d^a$ as \defi{extremal} vertices, and note that they have degree~1 in~$G(a,d)$, unless $a=d=1$, in which case the graph is a single vertex having degree~0.

\begin{proof}
We argue by induction on~$a$ and~$d$.
For $a=1$ and any $d\geq 1$, the graph~$G(a,d)$ is the path on $d$ vertices, so the claim is trivially true.
For the induction step let~$a\geq 2$.
We split~$G(a,d)$ into subgraphs $H_i$ for $i=1,\ldots,d$ where $H_i$ contains all vertices in which the first coordinate equals~$i$.
Note that $H_i\simeq G(a-1,d-(i-1))$ for all $i=1,\ldots,d$, in particular $H_d\simeq G(a-1,1)$ is a single vertex.
By induction, $H_i$ has a spanning tree~$T_i$ with $\Delta(T_i)\leq 3$ for all $i=1,\ldots,d$.
Furthermore, the two extremal vertices have degree~1 in~$T_i$ for $i=1,\ldots,d-1$ and degree~0 in~$T_d$.
We join the trees~$T_i$ to a single spanning tree~$T$ of~$G(a,d)$ by adding the edges~$\big((i,d^{a-1}),(i+1,d^{a-1})\big)$ for $i=1,\ldots,d-1$ between their extremal vertices.
\end{proof}

\section{Colorful triangulations}
\label{sec:colorful}

In this section we consider the setting of colorful triangulations introduced by Sagan, with the goal of proving Theorems~\ref{thm:FN-ham} and~\ref{thm:Falpha-ham}.

\subsection{Alternating colors}
\label{sec:alternating}

We first assume that the number~$N$ of points is even and the coloring sequence is~$\alpha=1^N$, i.e., the coloring pattern along the points $1,\ldots,N$ is $\tr\tb\tr\tb\cdots \tr\tb=(\tr\tb)^{N/2}$.
Recall that~$\cC_N$ denotes the set of all colorful triangulations with this coloring sequence.

\begin{figure}
\includegraphics[page=2]{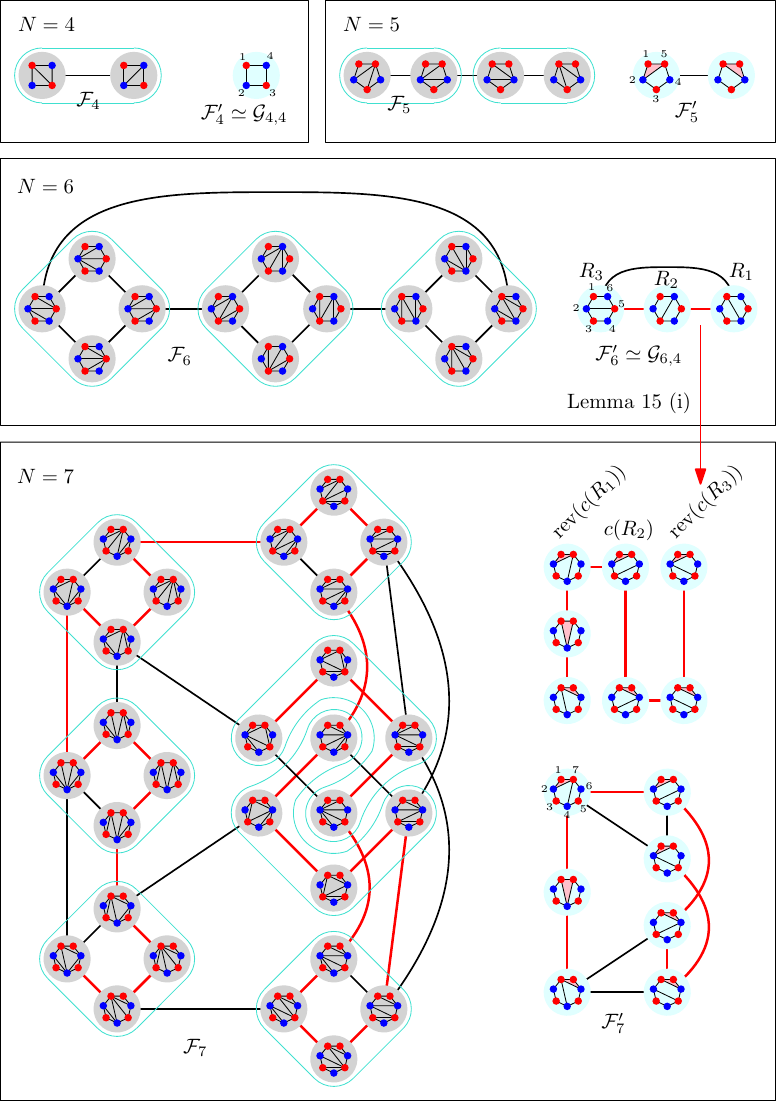}
\caption{Flip graphs of colorful triangulations and reduced graphs for $N=4,5,6,7$.
A Hamilton path in~$\cF_7$ is highlighted.}
\label{fig:F4567}
\end{figure}

\torsten{Make sure lemma references in the pictures are up-to-date}

\begin{figure}
\includegraphics[page=3]{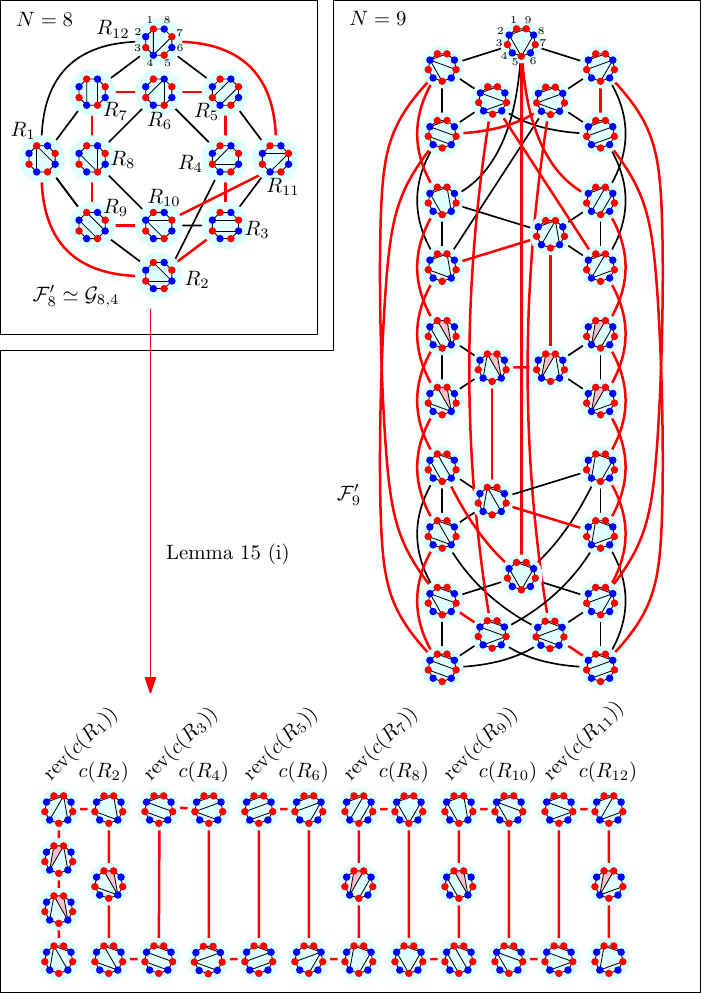}
\caption{Reduced graphs of colorful triangulations for $N=8,9$.}
\label{fig:F89}
\end{figure}

Let $T\in\cC_N$ be a colorful triangulation.
We say that an edge of~$T$ is \defi{monochromatic} if both endpoints have the same color, and we say that it is \defi{colorful} if both endpoints have distinct colors.
We observe the following:
\begin{enumerate}[label=(\roman*),leftmargin=8mm, noitemsep, topsep=3pt plus 3pt]
\item Every triangle of~$T$ has exactly one monochromatic edge.
\item Every monochromatic edge of~$T$ is an inner edge.
\end{enumerate}
Consequently, if we remove from~$T$ all monochromatic edges, keeping only the colorful ones, then the resulting dissection~$r(T)$ is a quadrangulation on the point set.
Indeed, by (i) every triangle is destroyed, and by (i)+(ii) destroying a triangle creates a quadrangle.
While $T$ has $N-2=n$ triangles, $r(T)$ has $q:=(N-2)/2=n/2$ quadrangles.
Furthermore, there are $2^q$ many colorful triangulations that yield the same quadrangulation~$r(T)$ by removing monochromatic edges.
They are obtained from~$r(T)$ by placing a diagonal in each of the $q$ quadrangles in one of the two ways.
Note that the subgraph of~$\cF_N$ induced by those $2^q$ triangulations is isomorphic to the $q$-dimensional hypercube~$Q_q$, as each of the $q$ monochromatic edges in~$T$ can be flipped independently from the others.
We thus obtain a partition of~$\cF_N$ into hypercubes~$Q_q$, plus edges between them.
These copies of hypercubes are highlighted by blue bubbles in Figure~\ref{fig:F4567}.

We also note that every quadrangulation~$R$ on $N$ points equals~$r(T)$ for some colorful triangulation~$T\in\cC_N$.
Indeed, given $R$, then coloring the $N$ points red and blue alternatingly will make all edges colorful.
We define a reduced graph~$\cF_N'$ that has as nodes all quadrangulations on $N$ points, and for any two colorful triangulations~$T$ and~$T'$ that differ in a flip of a colorful edge, we add an edge between~$r(T)$ and~$r(T')$ in~$\cF_N'$; see Figures~\ref{fig:F4567} and~\ref{fig:F89}.
We observe that $\cF_N'$ is isomorphic to the flip graph of quadrangulations~$\cG_{N,4}$, i.e., we have $\cF_N'\simeq \cG_{N,4}$.
These arguments yield a direct combinatorial proof for the first equality in Theorem~\ref{thm:count}.

\begin{lemma}
\label{lem:FN-ham-even}
Let $q\geq 3$ be an integer and $N:=2q+2$, and let $\cS$ be a spanning tree of~$\cF_N'$.
\begin{enumerate}[label=(\roman*),leftmargin=8mm, noitemsep, topsep=1pt plus 1pt]
\item If $q=3$ and $\Delta(\cS)\leq 2$ (i.e., $\cS$ is a Hamilton path), then $\cF_N$ has a Hamilton cycle.
\item If $q\geq 4$ and $\Delta(\cS)\leq 3$, then $\cF_N$ has a Hamilton cycle.
\end{enumerate}
\end{lemma}

The conclusions of the lemma do not hold when~$q=2$:
Indeed, while $\cF_6'$ has a Hamilton path (the graph is a triangle), there is no Hamilton path or cycle in~$\cF_6$; see Figure~\ref{fig:F4567}.

\begin{proof}
The idea is to `uncompress' the spanning tree~$\cS$ in~$\cF_N'$ to a Hamilton cycle in~$\cF_N$.
Specifically, for any quadrangulation~$R$ we consider the $2^q$ colorful triangulations~$C(R):=\{T\in\cC_N\mid r(T)=R\}$, and we let $Q(R)$ denote the subgraph of~$\cF_N$ spanned by the triangulations in~$C(R)$.
Recall that $Q(R)\simeq Q_q$, i.e., $Q(R)$ is isomorphic to the $q$-dimensional hypercube.
In the first step of the uncompression, we replace each edge $(R,R')$ of~$\cS$ by two edges~$(T_1,T_1')$, $(T_2,T_2')$ with $T_1,T_2\in C(R)$ and $T_1',T_2'\in C(R')$.
We refer to the edges $(T_1,T_1')$ and~$(T_2,T_2')$ as \defi{connectors}, and to their end nodes~$T_1,T_2,T_1',T_2'$ as \defi{terminals}.
In the second step, each quadrangulation~$R$ with degree~$d$ in~$\cS$ is replaced by $d$ paths that together visit all nodes in~$Q(R)$ and which join the connectors at their terminals to a single Hamilton cycle.

We now describe both steps in detail.
For a given quadrangulation~$R$, we associate each of the colorful triangulations~$T\in C(R)$ by a bitstring~$b(T)\in\{0,1\}^q$ as follows:
We label the $q$ quadrangles of~$R$ arbitrarily by $j=1,\ldots,q$, and we define $b(T)_j:=0$ if the monochromatic edge of~$T$ that sits inside the $j$th quadrangle of~$R$ connects the two red points, and otherwise (if it connects two blue points) $b(T)_j:=1$.

\begin{figure}[h!]
\includegraphics[page=5]{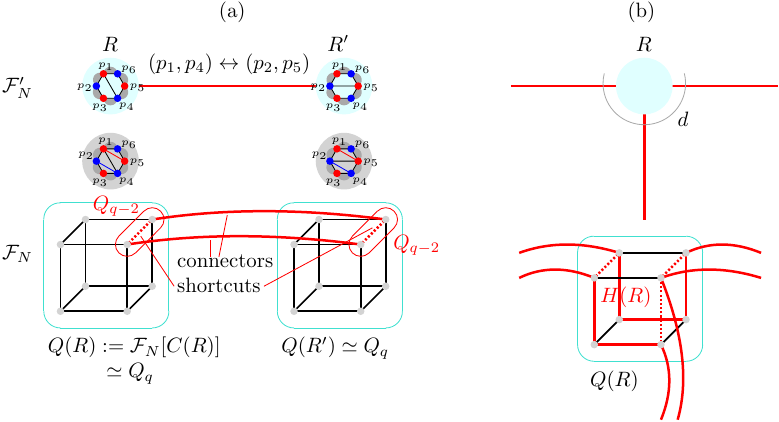}
\caption{Illustration of the proof of Lemma~\ref{lem:FN-ham-even}.}
\label{fig:FN-ham}
\end{figure}

Now consider an edge~$(R,R')$ of~$\cS$, which we aim to replace by two connectors~$(T_1,T_1')$, $(T_2,T_2')$ with $T_1,T_2\in C(R)$ and $T_1',T_2'\in C(R')$.
We denote the edge flipped in~$R$ by~$(p_1,p_4)$, and we label the points of the adjacent 4-gons in circular order by~$p_1,p_2,p_3,p_4$ and~$p_4,p_5,p_6,p_1$, respectively, such that the edge~$(p_1,p_4)$ is replaced by~$(p_2,p_5)$; see Figure~\ref{fig:FN-ham}~(a).
It follows that $T_1,T_1',T_2,T_2'$ must be triangulations that contain the two monochromatic edges~$(p_2,p_4)$ and~$(p_5,p_1)$ (but neither $(p_1,p_3)$ nor~$(p_4,p_6)$).
Consequently, the two corresponding bits of~$b(T_1),b(T_1'),b(T_2),b(T_2')$ must have a prescribed value, and therefore $T_1,T_2$ and $T_1',T_2'$ can be chosen from a $(q-2)$-dimensional subcube of~$Q(R)$ and~$Q(R')$, respectively.
Also, we will choose the connectors so that the pairs of terminals $(T_1,T_2)$ and $(T_1',T_2')$ differ only in a single flip, i.e., we select the two pairs of terminals as edges in their respective cubes, and we call these edges in~$Q(R)$ and~$Q(R')$ \defi{shortcuts}.
Note that the two connectors with the two shortcuts form the 4-cycle~$(T_1,T_1',T_2',T_2)$.
By the assumption $q\geq 3$ we have $q-2\geq 1$, i.e., there is at least one choice for each prescribed edge.

If the node~$R$ has degree~$d$ in~$\cS$, then we have to choose $d$ distinct shortcut edges in the hypercube~$Q(R)$, each selected from a distinct (but not necessarily disjoint) $(q-2)$-dimensional subcube, and to find a Hamilton cycle~$H(R)$ in~$Q(R)$ that contains all of these edges; see Figure~\ref{fig:FN-ham}~(b).
By Lemma~\ref{lem:edge-ham}, it is enough to ensure that the shortcut edges together form paths in~$Q(R)$.
If $d\leq 2$ (case~(i) of the lemma), then this is clear, as one or two edges always form one or two paths.
If $d=3$ (case~(ii) of the lemma), one has to avoid that all three shortcut edges are incident to the same node, which is easily possible under the stronger assumption~$q\geq 4$.

Then the Hamilton cycle in~$\cF_N$ is obtained by taking the symmetric difference of the edge sets of the cycles~$H(R)\seq Q(R)$ for all quadrangulations~$R$ on $N$ points with the 4-cycles formed by the connectors and shortcuts (i.e., the shortcuts are removed, and the connectors are added instead).
This completes the proof.
\end{proof}

\subsection{General coloring patterns}
\label{sec:general}

We now consider an arbitrary coloring sequence~$\alpha=(\alpha_1,\ldots,\alpha_\ell)$ and the corresponding coloring pattern defined in~\eqref{eq:col-pat}.
Recall that $\cC_\alpha$ denotes the set of all colorful triangulations with this coloring pattern, and that the corresponding flip graph is denoted by~$\cF_\alpha$.
The graph~$\cF_\alpha$ is an induced subgraph of the associahedron~$\cG_N$, where $N=\sum_{i=1}^\ell \alpha_i$.
As in the previous section, a colorful triangulation~$T\in\cC_\alpha$ has two types of edges, namely monochromatic and colorful edges.
We write $E_\alpha$ for the set of boundary edges that are monochromatic in~$T$, i.e., these are the pairs of points~$(i,i+1)$ for $i=1,\ldots,N$ (modulo $N$) where both endpoints receive the same color.
Generalizing the discussion from the previous section, we observe the following:
\begin{enumerate}[label=(\roman*),leftmargin=8mm, noitemsep, topsep=3pt plus 3pt]
\item Every triangle of~$T$ has exactly one monochromatic edge.
\item Apart from the edges in~$E_\alpha$, every monochromatic edge of~$T$ is an inner edge.
\end{enumerate}
Consequently, if we remove from~$T$ all monochromatic inner edges (the edges in~$E_\alpha$ are boundary edges and hence not removed), keeping only the colorful ones, then the resulting dissection~$r(T)$ has $t:=N-\ell$ triangles that contain the edges in~$E_\alpha$ and $q:=(\ell-2)/2$ quadrangles.
Furthermore, there are $2^q$ many colorful triangulations that yield the same dissection~$r(T)$ by removing monochromatic edges.
They are obtained from~$r(T)$ by placing a diagonal in each of the $q$ quadrangles in one of the two ways.
Note that the subgraph of~$\cF_\alpha$ induced by those~$2^q$ triangulations is isomorphic to the $q$-dimensional hypercube~$Q_q$.
We thus obtain a partition of~$\cF_\alpha$ into hypercubes~$Q_q$, plus edges between them.

We refer to a dissection of a convex $N$-gon into $q$ quadrangles and $t$ triangles that contain all the edges of~$E_\alpha$ as an \defi{$\alpha$-angulation}, and we write $\cD_\alpha$ for the set of all such dissections.
We also note that every $\alpha$-angulation~$R$ on $N$ points equals~$r(T)$ for some colorful triangulation~$T\in\cC_\alpha$.
Indeed, given $R$, then coloring the $N$ points according to the pattern in~\eqref{eq:col-pat} will make all edges apart from the ones in~$E_\alpha$ colorful.
We define a reduced graph~$\cF_\alpha'$ that has as nodes all $\alpha$-angulations on $N$ points, and for any two colorful triangulations $T$ and~$T'$ that differ in a flip of a colorful edge, we add an edge between~$r(T)$ and~$r(T')$ in~$\cF_\alpha'$; see Figures~\ref{fig:F4567}, \ref{fig:F89} and~\ref{fig:Falpha}.

\begin{figure}[h!]
\includegraphics[page=4]{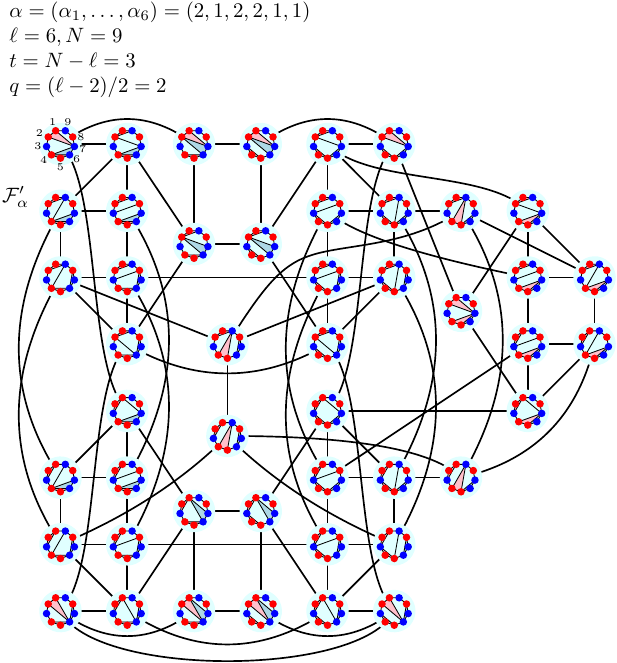}
\caption{Reduced graph of colorful triangulations for the coloring sequence $\alpha=(2,1,2,2,1,1)$ (\tr\tr\tb\tr\tr\tb\tb\tr\tb).}
\label{fig:Falpha}
\end{figure}

The proof of Lemma~\ref{lem:FN-ham-even} presented in the previous section generalizes straightforwardly, yielding the following statement.
Note that the variable~$\ell$ in Lemma~\ref{lem:FN-ham-alpha} below plays the role of~$N=2q+2$ in Lemma~\ref{lem:FN-ham-even}, and so the assumptions $\ell=8$ and~$\ell\geq 10$ translate to $q=(\ell-2)/2=3$ and~$q\geq 4$ used in the proof of Lemma~\ref{lem:FN-ham-even}, respectively.

\begin{lemma}
\label{lem:FN-ham-alpha}
Let $\alpha=(\alpha_1,\ldots,\alpha_\ell)$ be a coloring sequence of (even) length~$\ell\geq 8$, and let $\cS$ be a spanning tree of~$\cF_\alpha'$.
\begin{enumerate}[label=(\roman*),leftmargin=8mm, noitemsep, topsep=1pt plus 1pt]
\item If $\ell=8$ and $\Delta(\cS)\leq 2$ (i.e., $\cS$ is a Hamilton path), then $\cF_\alpha$ has a Hamilton cycle.
\item If $\ell\geq 10$ and $\Delta(\cS)\leq 3$, then $\cF_\alpha$ has a Hamilton cycle.
\end{enumerate}
\end{lemma}

The next lemma allows us to duplicate the occurrence of a color that appears only once (i.e., we change $\alpha_i=1$ to some larger number $\alpha_i>1$), while inductively maintaining spanning trees with small degrees in the corresponding reduced flip graphs.

\begin{lemma}
\label{lem:zigzag}
Let $\beta=(\beta_1,\ldots,\beta_\ell)$ and $\alpha=(\alpha_1,\ldots,\alpha_\ell)$ be coloring sequences of (even) length~$\ell\geq 4$ that agree in all but the $i$th entry such that $\beta_i=1$ and $\alpha_i>1$.
\begin{enumerate}[label=(\roman*),leftmargin=8mm, noitemsep, topsep=1pt plus 1pt]
\item If $\cF_\beta'$ has a Hamilton path and $\alpha_i=2$, then $\cF_\alpha'$ has a Hamilton path.
\item If $\cF_\beta'$ has a spanning tree~$\cT$ with $\Delta(\cT)\leq 3$, then $\cF_\alpha'$ has a spanning tree~$\cS$ with $\Delta(\cS)\leq 3$.
\end{enumerate}
\end{lemma}

In Figure~\ref{fig:F4567}, part~(i) of this lemma is applied to construct a Hamilton path in~$\cF_7'$ from one in~$\cF_6'$.
Similarly, in Figure~\ref{fig:F89}, a Hamilton path in~$\cF_9'$ is constructed from one in~$\cF_8'$.

The idea for the proof of part~(i) is the same as the one used by Hurtado and Noy~\cite{MR1723053}.

\begin{proof}
We consider the point $p:=\sum_{j=1}^i \beta_j$ on the boundary, which is neighbored by two points~$p-1$ and~$p+1$ (modulo~$N=\sum_{j=1}^\ell \beta_j$) of the opposite color.
We also define $a:=\alpha_i-1$, i.e., we want to add $a$ points of the same color as $p$ next to $p$.
Let $R\in \cD_\beta$ be a $\beta$-angulation, and let $(p,q_1),(p,q_2),\ldots,(p,q_d)$ be the edges incident with the point~$p$ in~$R$ in counterclockwise order (all these edges are colorful), such that $q_1=p+1$ and $q_d=p-1$; see Figure~\ref{fig:zigzag}.
If $a=1$, then for $j=1,\ldots,d$ we let $R^j$ be the $\alpha$-angulation obtained from~$R$ by inflating the edge~$(p,q_j)$ to a triangle~$(p,p',q_j)$.
Specifically, the single point~$p$ is split into two consecutive points~$p$ and~$p'$ on the boundary joined by an edge, and $q_1,\ldots,q_j$ remain connected to~$p'$, whereas $q_j,q_{j+1},\ldots,q_d$ remain connected to~$p$.
More generally, we define $J(R):=\{(j_1,\ldots,j_a)\mid 1\leq j_1\leq j_2\leq \cdots \leq j_a\leq d\}$, $\jcheck(R):=1^a$ and $\jhat(R):=d^a$, and for any $(j_1,j_2,\ldots,j_a)\in J(R)$ we let $R^{(j_1,\ldots,j_a)}$ be the $\beta$-angulation obtained from~$R$ by inflating each of the edges~$(p,q_{j_1}),\ldots,(p,q_{j_a})$ to a triangle.
Note that the same edge may be inflated multiple times; see the bottom rows with labels $a=2$ and~$a=3$ in Figure~\ref{fig:zigzag}.
Specifically, if some value~$j_b$, $b\in\{1,\ldots,a\}$, appears $c$ times in the list $j_1,\ldots,j_a$, then the edge $(p,q_{j_b})$ is inflated to $c$ many triangles.
Furthermore, observe that $R^{(j_1,\ldots,j_a)}$ differs from~$R^{(j_1',\ldots,j_a')}$ in a flip if and only if $(j_1,\ldots,j_a)$ and $(j_1',\ldots,j_a')$ differ in a single entry by~$\pm 1$, i.e., the subgraph of~$\cF_\alpha'$ induced by the $\alpha$-angulations $R^{(j_1,\ldots,j_a)}$, $(j_1,\ldots,j_a)\in J(R)$, is isomorphic to the graph~$G(a,d)$ defined in Section~\ref{sec:graphs}.
By Lemma~\ref{lem:Gad}, it admits a spanning tree~$\cS(R)$ with $\Delta(\cS(R))\leq 3$, in which the nodes $R^{\jcheck(R)}$ and $R^{\jhat(R)}$ have degree~1.
If $a=1$, then this subgraph and spanning tree is simply a path, and we refer to it as children sequence $c(R):=(R^1,R^2,\ldots,R^d)$.
Also note that if $(R,Q)$ is an edge in~$\cF_\beta'$, then $(R^{\jcheck(R)},Q^{\jcheck(Q)})$ and~$(R^{\jhat(R)},Q^{\jhat(Q)})$ are both edges in~$\cF_\alpha'$.

\begin{figure}[t!]
\includegraphics[page=6]{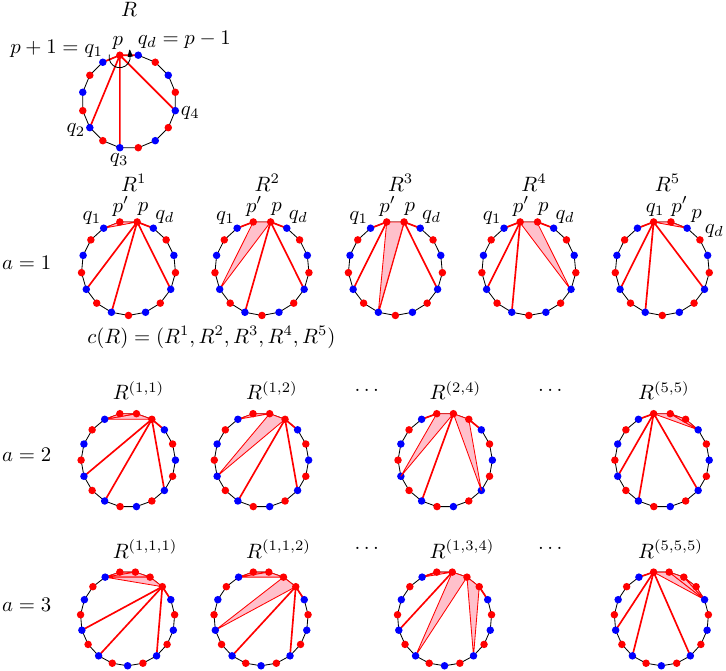}
\caption{Illustration of the proof of Lemma~\ref{lem:zigzag}.
Edges of the $\beta$-angulation~$R$ that are not incident to the point~$p$ are not shown for clarity.}
\label{fig:zigzag}
\end{figure}

\begin{figure}[h!]
\includegraphics[page=7]{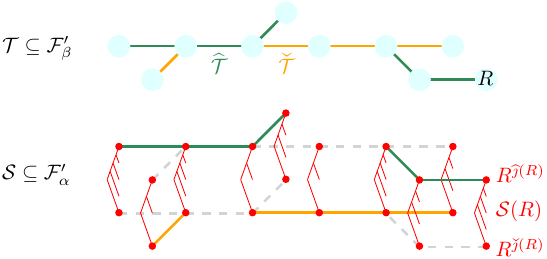}
\caption{Illustration of the proof of part~(ii) of Lemma~\ref{lem:zigzag}.
The fork-like structures are the spanning trees obtained from Lemma~\ref{lem:Gad} (cf.~Figure~\ref{fig:Gad}).}
\label{fig:zigzag2}
\end{figure}
\end{proof}

We now prove~(i), using the assumption $\alpha_i=2$, i.e., $a=1$.
Let $P=(R_1,\ldots,R_L)$ be a Hamilton path in~$\cF_\beta'$.
Then a Hamilton path in~$\cF_\alpha'$ is given by $P':=(\rev(c(R_1)),c(R_2),\rev(c(R_3)),c(R_4),\ldots)$; see Figures~\ref{fig:F4567} and~\ref{fig:F89}.

For proving~(ii), let $\cT$ be a spanning tree in~$\cF_\beta'$ with $\Delta(\cT)\leq 3$.
We partition its edges into two disjoint forests of paths~$\widecheck{\cT}$ and~$\widehat{\cT}$, i.e., we have $\Delta(\widecheck{\cT})\leq 2$ and $\Delta(\widehat{\cT})\leq 2$; see Figure~\ref{fig:zigzag2}.
We then define the spanning tree~$\cS$ as the union of the trees~$\cS(R)$ for all $\beta$-angulations $R$ plus the edges $\{(R^{\jcheck(R)},Q^{\jcheck(Q)})\mid (R,Q)\in\widecheck{\cT}\}$ and $\{(R^{\jhat(R)},Q^{\jhat(Q)})\mid (R,Q)\in\widehat{\cT}\}$.
It is easy to check that~$\cS$ is indeed a spanning tree of~$\cF_\alpha'$ with~$\Delta(\cS)\leq 3$.

\subsection{Proofs of Theorems~\ref{thm:FN-ham} and~\ref{thm:Falpha-ham}}

\begin{proof}[Proof of Theorem~\ref{thm:Falpha-ham}]
For the given coloring sequence~$\alpha=(\alpha_1,\ldots,\alpha_\ell)$ of length~$\ell\geq 10$, we consider the alternating coloring sequence~$\beta=1^\ell$ of length~$\ell$, i.e., all repetitions of colors are reduced to a single occurrence, and $\beta$ corresponds to coloring~$\ell$ points alternatingly red and blue.
The corresponding reduced flip graph~$\cF_\ell'$ is isomorphic to the rotation graph of ternary trees, i.e., we have $\cF_\beta'=\cF_\ell'\simeq \cG_{\ell,4}$.
Theorem~\ref{thm:GNk-ham} yields a Hamilton path in the graph~$\cF_\beta'=\cF_\ell'$.
Applying Lemma~\ref{lem:zigzag}~(ii) once for each~$\alpha_i$ with $\alpha_i>1$, we obtain that~$\cF_\alpha'$ has a spanning tree~$\cS$ with~$\Delta(\cS)\leq 3$.
Lastly, applying Lemma~\ref{lem:FN-ham-alpha}~(ii) yields that~$\cF_\alpha$ has a Hamilton cycle.
\end{proof}

\begin{proof}[Proof of Theorem~\ref{thm:FN-ham}]
For $N\geq 10$ the result is a special case of Theorem~\ref{thm:Falpha-ham}, so it remains to cover the cases~$N=8$ and~$N=9$.
A Hamilton path~$P$ in~$\cF_8'$ is guaranteed by Theorem~\ref{thm:GNk-ham}; see Figure~\ref{fig:F89}.
Applying Lemma~\ref{lem:FN-ham-even}~(i) to~$P$ yields that~$\cF_8$ has a Hamilton cycle.
Applying Lemma~\ref{lem:zigzag}~(i) to~$P$ proves that~$\cF_9'$ has a Hamilton path~$P'$.
Applying Lemma~\ref{lem:FN-ham-even}~(i) to~$P'$ shows that~$\cF_9$ has a Hamilton cycle.
\end{proof}

\section{Proof of Theorem~\ref{thm:a11111}}
\label{sec:a11111}

\begin{figure}[h!]
\includegraphics[page=10]{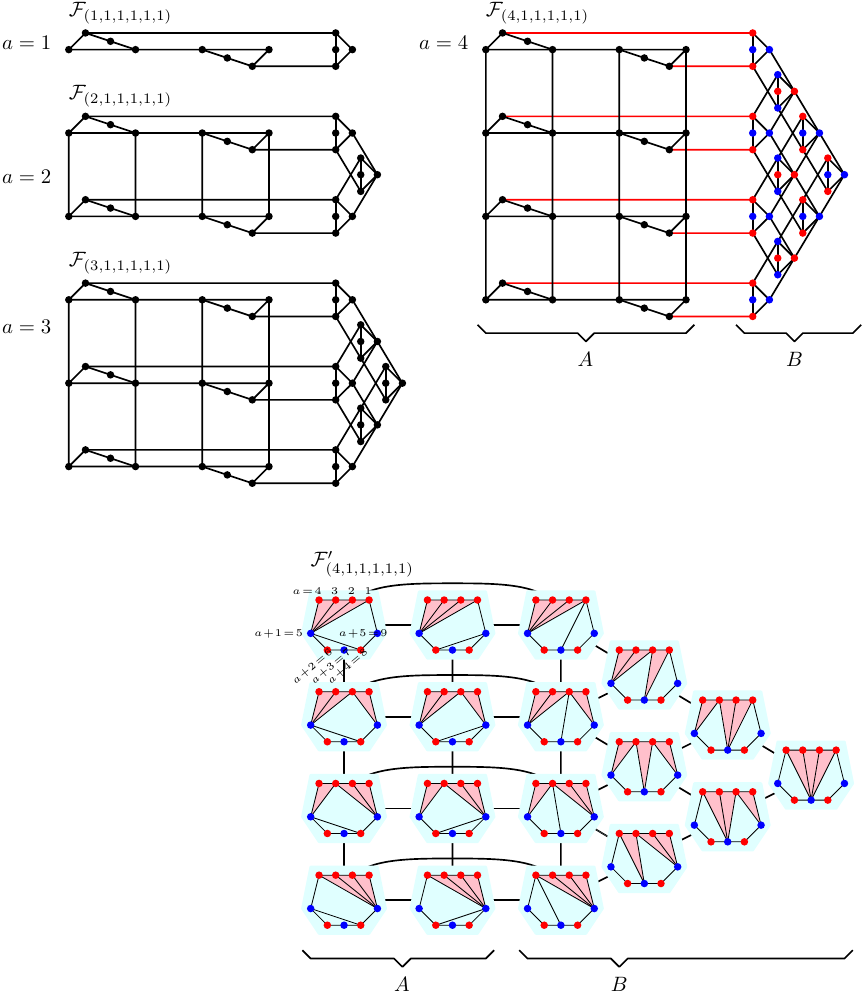}
\caption{Top: Illustration of the family of graphs~$\cF_{(a,1,1,1,1,1)}$, with the subgraphs~$A$ and~$B$ and the two partition classes of~$B$ highlighted.
Bottom: The reduced graph~$\cF_{(a,1,1,1,1,1)}'$ for $a=4$, explaining the general structure.}
\label{fig:a11111}
\end{figure}

\begin{figure}[h!]
\includegraphics[page=13]{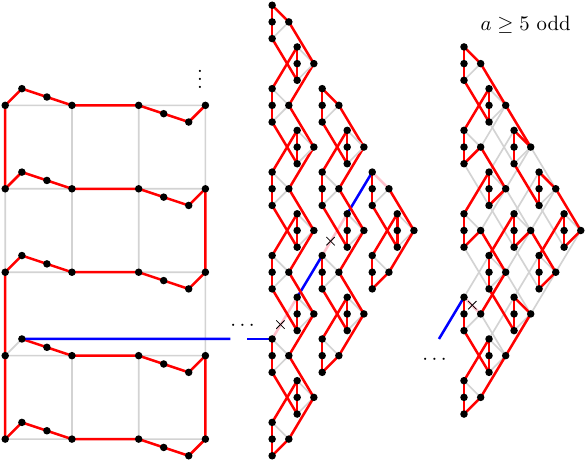}
\caption{Construction of a Hamilton path in~$\cF_{(a,1,1,1,1,1)}$ for $a\geq 4$.}
\label{fig:a11111-path}
\end{figure}

\begin{proof}[Proof of Theorem~\ref{thm:a11111}]
We first argue that~$\cF_\alpha$, $a\geq 1$, has no Hamilton cycle.
For this observe that the graph has two subgraphs~$A$ and~$B$, as shown in Figure~\ref{fig:a11111}.
Specifically, the subgraph~$A$ is induced by all triangulations~$T\in\cC_\alpha$ in which the point~$a+3$ has degree~2 in the dissection~$r(T)$ (i.e., it only has the previous point~$a+2$ and the next point~$a+4$ as neighbors).
The subgraph~$B$ is induced by all other triangulations from~$\cC_\alpha$, and the structure and interaction of both subgraphs can be deduced from the bottom part of the figure.  
The subgraph~$B$ is bipartite with two partition classes of equal sizes, and it is connected to~$A$ only via edges starting from the same partition class.
Consequently, if there were a Hamilton cycle in~$\cF_\alpha$, then removing the edges between~$A$ and~$B$ from the cycle would leave a set of spanning paths in~$B$ that all start and end in the same partition class of~$B$, a contradiction, as the partition classes of~$B$ have the same size.

For $a=1$ it is easy to see that there is no Hamilton path; see the top left part of Figure~\ref{fig:a11111} or Figure~\ref{fig:F4567} (the graph~$\cF_6$).
For $a=2$ a Hamilton path is shown in Figure~\ref{fig:F4567} (in the graph~$\cF_7$).
For $a=3$ a computer search shows that there is no Hamilton path.
For $a\geq 4$, a Hamilton path can be constructed by combining the gadgets shown in Figure~\ref{fig:a11111-path}.
\end{proof}

\section{A unified proof of Theorem~\ref{thm:GNk-ham}}
\label{sec:unified}

As a building block for our algorithmic constructions, in this section we describe a proof of Theorem~\ref{thm:GNk-ham} that is simpler than the original proof by Huemer, Hurtado, and Pfeifle~\cite{MR2474724}, and that unifies it with the Hamiltonicity proof of the associahedron given by Hurtado and Noy~\cite{MR1723053} (the case $k=2$).
Also, the new construction is an application of the `zigzag' principle developed in~\cite{MR4391718}.

\subsection{Tree rotation}

We use the following compact notation for $k$-ary trees, defined inductively as follows:
The unique empty tree is denoted by~$\varepsilon$ (it has 0 vertices).
If $T_1,\ldots,T_k$ are $k$-ary trees, then $T:=[T_1,\ldots,T_k]$ represents the $k$-ary tree that has a root vertex with $T_1,\ldots,T_k$ as its subtrees.
Furthermore, the \defi{subtrees} of~$T$ are given by $T$ itself and by the subtrees of~$T_i$ for $i=1,\ldots,k$.
If $T_i=\cdots=T_j=\varepsilon$, then we write $T=[T_1,\ldots,T_k]$ as $[T_1,\ldots,T_{i-1},\varepsilon^{j-i+1},T_{j+1},\ldots,T_k]$.
In particular, $[\varepsilon^k]$ is the unique 1-vertex $k$-ary tree.

Using this notation, tree rotation becomes the following operation; see Figures~\ref{fig:flip2} and~\ref{fig:flip3}.
Given trees $T_1,\ldots,T_{2k-1}$, we define
\begin{equation}
\label{eq:Ri}
R^i:=[T_1,\ldots,T_{i-1},[T_i,\ldots,T_{i+k-1}],T_{i+k},\ldots,T_{2k-1}]
\end{equation}
for $i=1,\ldots,k$.
We write $R\sim R'$ for two trees that differ only in subtrees~$R^i$ and~$R^j$ defined before for integers~$1\leq i<j\leq k$, and we say that $R$ and~$R'$ differ in a \defi{rotation}.
Note that for $k=2$ we have $R\sim R'$ if $R$ and~$R'$ differ only in subtrees~$R^1=[[A,B],C]$ and~$R^2=[A,[B,C]]$, respectively.
Similarly, for $k=3$ we have $R\sim R'\sim R''$ if $R$, $R'$ and~$R''$ differ only in subtrees~$R^1=[[A,B,C],D,E]$, $R^2=[A,[B,C,D],E]$, and $R^3=[A,B,[C,D,E]]$, respectively.

\subsection{Tree rotation with vertex labels}

For $k=2$ we are in the special case of binary trees.
The $n$ vertices of a binary tree~$T$ can be labeled uniquely by $1,\ldots,n$ using the search tree property:
I.e., if a vertex receives label~$a$, then all left descendants have labels strictly smaller than~$a$, and all right descendants have labels strictly larger than~$a$.
Specifically, for $T=[A,B]$ we give the labels $1,\ldots,v(A)$ to $A$, the label $v(A)+1$ to the root of~$T$, and the labels $v(A)+2,\ldots,v(A)+v(B)+1=n$ to~$B$; see Figure~\ref{fig:rot-vertex}~(a).
Then under the rotation operation $[[A,B],C]\sim [A,[B,C]]$ the subtrees $A,B,C$ are assigned the same sets of labels, and their parent vertices are assigned the labels~$a$ and $b$, respectively, where $a:=v(A)+1$ and $b:=v(A)+v(B)+2$.
We can thus think of a tree rotation as a rotation of the edge~$(a,b)$ in the vertex-labeled tree, in which the subtree~$B$ changes its parent.
Specifically, we refer to $[[A,B],C]\rightarrow [A,[B,C]]$ as \defi{down-rotation}, and to the inverse operation as \defi{up-rotation} of the larger vertex~$b$.

\begin{figure}[h!]
\includegraphics[page=2]{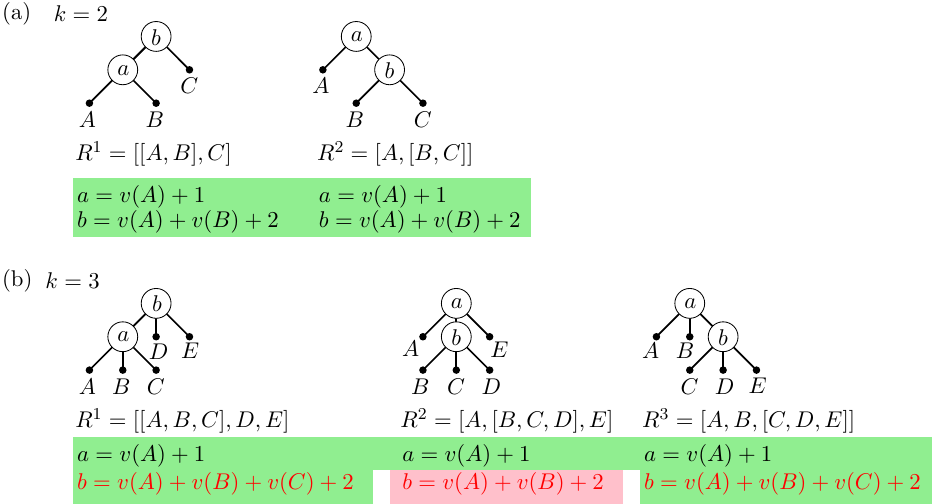}
\caption{Tree rotation of vertex-labeled trees.}
\label{fig:rot-vertex}
\end{figure}

Unfortunately, and maybe surprisingly, for $k\geq 3$ there is {\bf no} such general interpretation of $k$-ary tree rotation that would allow a consistent labeling of the vertices of the trees; see Figure~\ref{fig:rot-vertex}~(b).
Specifically, for $T=[T_1,\ldots,T_k]$ we may generalize the labeling from before as follows:
We give the labels $1,\ldots,v(T_1)$ to~$T_1$, the label $v(T_1)+1$ to the root of~$T$, and for $i=2,\ldots,k$, we give the labels $s_{i-1}+1,\ldots,s_i$ to~$T_i$, where for convenience we define $s_i:=\sum_{j=1}^i v(T_j)+1$.

Then for a tree~$R^i$ as in~\eqref{eq:Ri} the root receives the label~$s_k+1$ if $i=1$ and $s_1$ if~$i>1$.
Furthermore, the descendant of the root not belonging to any of the subtrees $T_1,\ldots,T_{2k-1}$ receives the label $s_1$ if~$i=1$ and $s_i+1$ if~$i>1$.
Note that $s_k+1=s_i+1$ if and only if $s_k=s_i$ if and only if $v(T_k)=v(T_{k-1})=\cdots=v(T_{i+1})=0$, i.e., the rightmost $k-i$ subtrees of $[T_1,\ldots,T_k]$ (this is the leftmost subtree of the root of~$R^1$) are all empty.
With this additional assumption, the aforementioned vertex labeling is preserved between pairs of trees from the set~$\{R^1,R^i,R^{i+1},\ldots,R^k\}$ under rotations of the edge~$(a,b)$, where $a:=s_1$ and $b:=s_k+1=s_i+1$.
We call such rotations \defi{clean}.
We refer to $R^1\rightarrow R^s$, $i\leq s\leq k$, as \defi{down-rotations} of the larger vertex~$b$ \defi{by $k-s+1$ steps}, and to the inverse operations as \defi{up-rotations by $k-s+1$ steps}.
Furthermore, we refer to $R^t\rightarrow R^s$ with $i\leq s<t\leq k$ as \defi{down-rotation} of~$b$ \defi{by $t-s$ steps}, and to the inverse operation as \defi{up-rotation by $t-s$ steps}.

\subsection{The children sequence of a tree}
\label{sec:cseq}

Given a tree~$T$ with $n-1$ vertices, in the following we define a sequence~$c(T)$ of trees with $n$ vertices, which are obtained by inserting an additional vertex into~$T$ at diffent positions of the rightmost branch starting from the root.
This is done in such a way that any two consecutive trees in the sequence~$c(T)$ differ in a tree rotation.
It turns out that all these rotations are clean, and the sequence~$c(T)$ is created by down-rotations of a vertex with the largest label~$n$, and the reverse sequence is created by up-rotations of the vertex~$n$; see Figure~\ref{fig:insertion} below.

The following definitions are illustrated in Figure~\ref{fig:sseq}.
Every nonempty subtree~$S$ of~$T=[T_1,\ldots,T_k]$ is associated with a \defi{subtree sequence}~$\nu_T(S)$, which is a finite sequence of integers from~$\{1,\ldots,k\}$ defined as $\nu_T(S)=\varepsilon$ if $T=S$ and $\nu_T(S)=(i,\nu_{T_i}(S))$ if $S$ is a subtree of~$T_i$.
The sequence~$\nu_T(S)$ describes the path from the root of~$T$ to the subtree~$S$, i.e., it gives the indices of children to which to descend in order to reach~$S$.
Given a subtree~$S$ of~$T$ and the sequence $\nu:=\nu_T(S)$, then we define $T_\nu:=S$.
In particular, $T_\varepsilon=T$, and $T_i$ is the $i$th subtree of the root of~$T$.

The \defi{rightmost branch} in~$T$, denoted $\rho(T)$, is the lexicographically largest subtree sequence~$\nu_T(S)$  with entries from $\{2,\ldots,k\}$ (i.e., 1 is excluded) among all subtrees~$S$ of~$T$.
It corresponds to the branch in~$T$ starting at the root and descending towards the rightmost nonempty child, unless it is the first child.

\begin{figure}[h!]
\makebox[0cm]{ % artificial box to center the picture
\includegraphics[page=3]{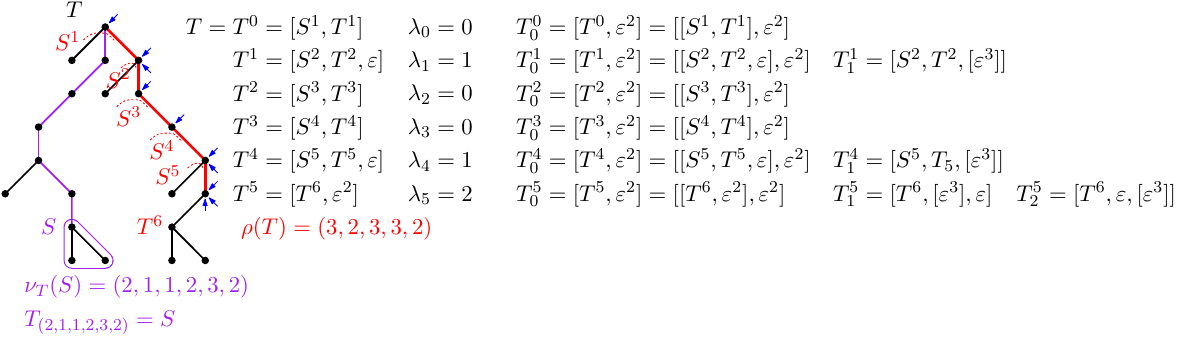}
}
\caption{Illustration of subtree sequences and the rightmost branch of a tree ($k=3$).
The sequence of all children created from~$T$ is displayed in Figure~\ref{fig:insertion}.}
\label{fig:sseq}
\end{figure}

Given the sequence $\rho(T)=(\rho_1,\ldots,\rho_\ell)$, we define a sequence of subtrees of~$T$ along its rightmost branch by $T^i:=T_{(\rho_1,\ldots,\rho_i)}$ for $i=0,\ldots,\ell$.
Furthermore, we define $T^{\ell+1}:=T_{\rho(T),1}$, which is the leftmost subtree of~$T_{\rho(T)}$.
For $i=0,\ldots,\ell$ we have $T^i=[S^{i+1},T^{i+1},\varepsilon^{\lambda_i}]$ for integers $\lambda_i$ and sequences of subtree~$S^{i+1}$ that are determined by~$T$.
Specifically, $T^{i+1}$ is the rightmost nonempty subtree of~$T^i$, $\lambda_i$ is the number of empty subtrees to the right of it, and $S^{i+1}$ is the sequence of subtrees to the left of it.
Note that $S^1,\ldots,S^\ell$ are nonempty sequences of subtrees, whereas $S^{\ell+1}=\varepsilon$, i.e., we have $0\leq \lambda_i<k-1$ if $i<\ell$ and $\lambda_\ell=k-1$.

We define $T^i_0:=[T^i,\varepsilon^{k-1}]$ for all $i=0,\ldots,\ell$, i.e., $T^i_0$ is obtained from~$T^i$ by adding a new root vertex and making~$T^i$ its leftmost child.
Furthermore, for $j=1,\ldots,\lambda_i$ we let $T^i_j$ be the tree obtained from~$T^i=[S^{i+1},T^{i+1},\varepsilon^{\lambda_i}]$ by replacing the $j$th occurrence of~$\varepsilon$ after~$T^{i+1}$ by $[\varepsilon^k]$.
In words, $T^i_j$ is obtained from~$T^i$ by inserting a new childless vertex $j$ positions to the right of the subtree~$T^{i+1}$.

For $i=0,\ldots,\ell$ and $j=0,\ldots,\lambda_i$ we let $c_{i,j}(T)$ be the tree obtained from~$T$ by replacing the subtree~$T^i$ by~$T^i_j$.
For $j=0$ this corresponds to inserting a new vertex between the subtree~$T^i$ and its parent vertex (as root if~$i=0$) and making~$T^i$ the first child of this new vertex.
For $j>0$ this corresponds to inserting a new vertex as a child of the root of the subtree~$T^i$, $j$ positions to the right of the rightmost nonempty subtree.

\begin{figure}[t!]
\includegraphics[page=4,scale=0.8]{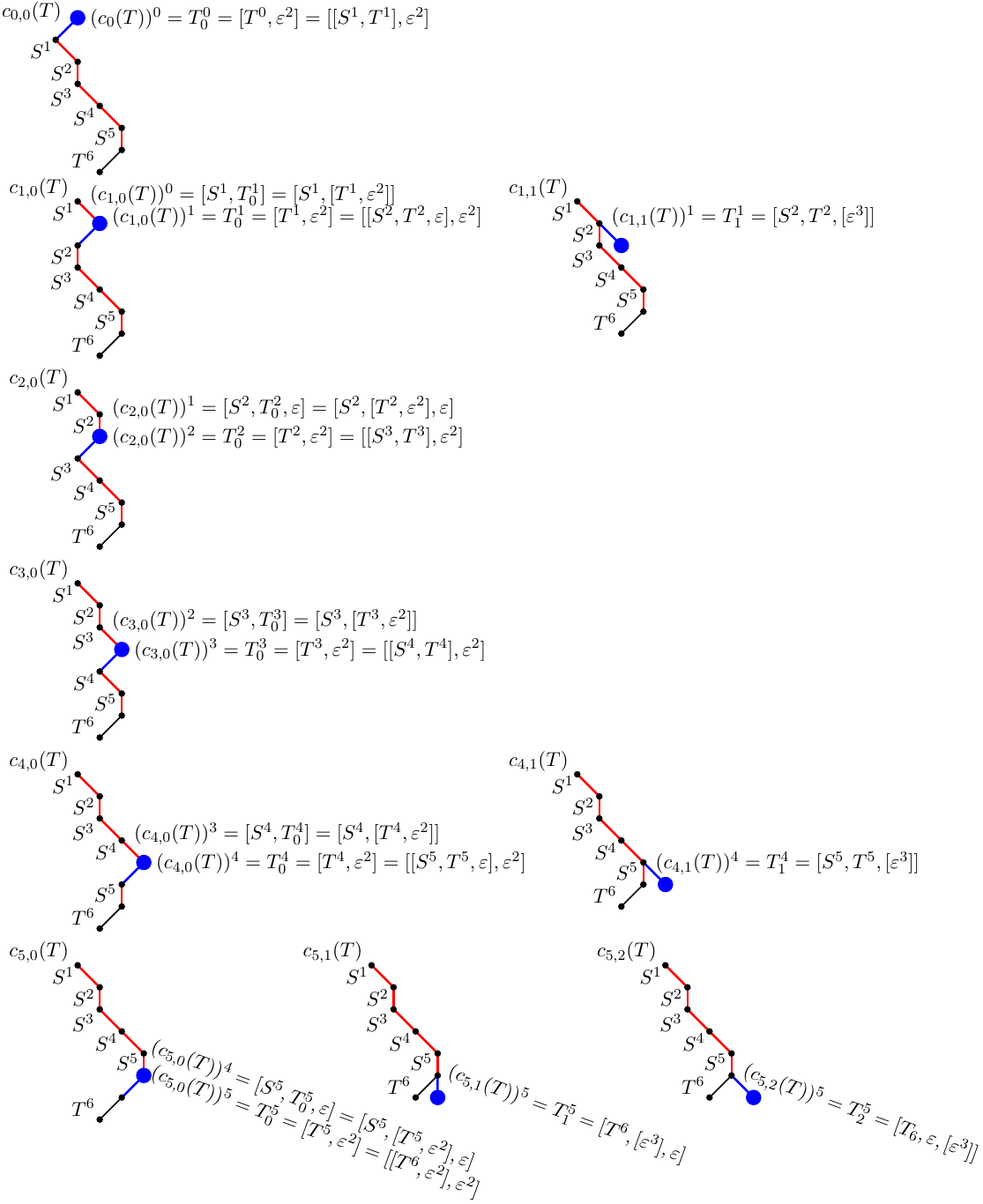}
\caption{Vertex insertion into the tree~$T$ from Figure~\ref{fig:sseq}. The newly inserted vertex is drawn bold.}
\label{fig:insertion}
\end{figure}

From these definitions we obtain
\begin{equation}
\label{eq:ci0i}
(c_{i,0}(T))^i=T^i_0=[T^i,\varepsilon^{k-1}]=[[S^{i+1},T^{i+1},\varepsilon^{\lambda_i}],\varepsilon^{k-1}]
\end{equation}
and
\begin{equation}
\label{eq:ci0im1}
(c_{i,0}(T))^{i-1}=[S^i,T^i_0,\varepsilon^{\lambda_{i-1}}]=[S^i,[T^i,\varepsilon^{k-1}],\varepsilon^{\lambda_{i-1}}];
\end{equation}
see Figure~\ref{fig:insertion}.
Furthermore, for $j=1,\ldots,\lambda_i$ we obtain
\begin{equation}
\label{eq:cij}
(c_{i,j}(T))^i=T^i_j=[S^{i+1},T^{i+1},\varepsilon^{j-1},[\varepsilon^k],\varepsilon^{\lambda_i-j}].
\end{equation}
For $i<\ell$ we obtain from \eqref{eq:ci0im1} that
\begin{equation}
\label{eq:cip10}
(c_{i+1,0}(T))^i=[S^{i+1},T^{i+1}_0,\varepsilon^{\lambda_i}]=[S^{i+1},[T^{i+1},\varepsilon^{k-1}],\varepsilon^{\lambda_i}].
\end{equation}
We define $C_i:=\{c_{i,0}(T),\ldots,c_{i,\lambda_i}(T)\}$.
Note that $|C_i|=\lambda_i+1$, in particular
\begin{equation}
\label{eq:Cell}
|C_\ell|=\lambda_\ell+1=k.
\end{equation}

Combining \eqref{eq:ci0i} and \eqref{eq:cij} shows that any two trees in $C_i$ differ in a tree rotation, and if $i<\ell$ then we see from~\eqref{eq:cip10} that any two trees in~$C_i\cup\{c_{i+1,0}(T)\}$ differ in a tree rotation; see Figure~\ref{fig:cseq}.
Observe that all these rotations are clean, i.e., if $T$ is a tree with vertex labels $1,\ldots,n-1$, then in the tree $c_{i,j}(T)$ the new largest label~$n$ is given to the root of the subtree~$T^i_0$ and to the rightmost leaf of $T^i_j$ for $j=1,\ldots,\lambda_i$.
I.e., we can think of the trees~$c_{i,j}(T)$ as the trees obtained from~$T$ by inserting the vertex~$n$ in all possible ways on the rightmost branch of~$T$ (and the other vertex labels remain unchanged).

\begin{figure}
\includegraphics[page=5]{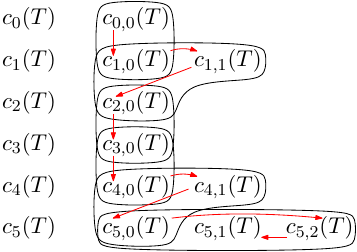}
\caption{Children sequence~$c(T)$ for the tree~$T$ from Figure~\ref{fig:sseq} (see also Figure~\ref{fig:insertion}).
Sets of trees surrounded by a bubble differ in a tree rotation.
The arrows correspond to down-rotations of the largest vertex~$n$ by 1 step.
}
\label{fig:cseq}
\end{figure}

We define the \defi{children sequence} of~$T$ as
\[ c(T):=(c_0(T),c_1(T),\ldots,c_\ell(T)) \]
with
\[ c_i(T):=c_{i,0}(T),c_{i,\lambda_i}(T),c_{i,\lambda_i-1}(T),\ldots,c_{i,1}(T); \]
see Figure~\ref{fig:cseq}.
Note that for every tree rotation along this sequence, at most two nonempty subtrees change parent.

\subsection{Rotations between first and last children}

To construct a Hamilton path in the graph~$\cG_{N,k+1}$, i.e., in the rotation graph of $n$-vertex $k$-ary trees, we will argue by induction on~$n$.
Given a Hamilton path ordering of all trees with $n-1$ vertices, we replace each tree by its children sequence obtained by inserting the new vertex~$n$.
Specifically, the children sequences will be traversed in forward or reverse order alternatingly, in a zigzag fashion; see Figure~\ref{fig:gc3}.
We now aim to understand how rotations are preserved when considering the first and last elements of children sequences $c(T)$ and~$c(T')$ for two trees~$T$ and~$T'$ that differ in a rotation.

For any sequence $x=(x_1,\ldots,x_\ell)$, we let $f(x):=x_1$ be its first element and $\ell(x):=x_\ell$ its last element.
More generally, we define $\ell^i(x):=x_{\ell+1-i}$, i.e., the $i$th-last entry of~$x$.
We now consider two trees $T$ and~$T'$ that differ in a clean tree rotation of an edge~$(a,b)$, with $a<b$, and we compare the first and last entries of the children sequences $c(T)$ and~$c(T')$.
For the rotation of the edge~$(a,b)$, we let $R^1,R^i,R^{i+1},\ldots,R^k$ denote the sequence of all trees obtained by clean rotations of the edge~$(a,b)$, such that $R^1\rightarrow R^k$ is a down-rotation of the vertex~$b$ by 1 step, and $R^k\rightarrow R^{k-1}\rightarrow\cdots\rightarrow R^i$ are all down-rotations of~$b$ by 1~step; see Figure~\ref{fig:child-rot}.
We let $T_1,\ldots,T_{2k-1}$ denote the corresponding subtrees of the edge~$(a,b)$; recall~\eqref{eq:Ri}.

It is easy to see that
\begin{equation}
\label{eq:frot}
f(c(T))\sim f(c(T')),
\end{equation}
i.e., the first trees in the sequences $c(T)$ and~$c(T')$ differ in the same rotation operation.
Furthermore, if one of $T_{k+1},\ldots,T_{2k-1}$ is nonempty then we have
\begin{subequations}
\begin{equation}
\label{eq:lrot1}
\ell(c(T))\sim \ell(c(T')),
\end{equation}
i.e., the last trees in the sequences $c(T)$ and~$c(T')$ differ in the same rotation operation.
However, if $T_{k+1}=\cdots=T_{2k-1}=\varepsilon$, then we have
\begin{equation}
\label{eq:lrot2}
\begin{split}
\ell(R^1) &\sim \ell^{k-s+1}(R^s) \text{ for } s=i,\ldots,k, \\
\ell(R^t) &\sim \ell^{t-s+1}(R^s) \text{ for } i\leq s<t\leq k.
\end{split}
\end{equation}
\end{subequations}

\begin{figure}[h!]
\makebox[0cm]{ % artificial box to center the picture
\includegraphics[page=6]{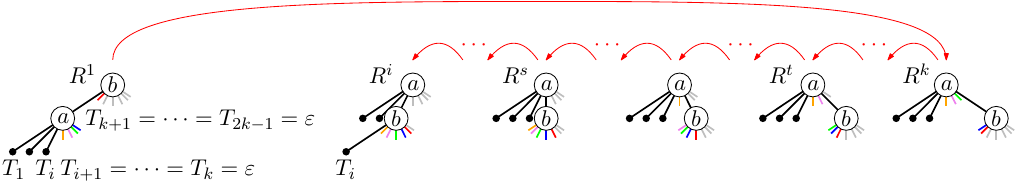}
}
\caption{Relation between last entries of~$c(T)$ and~$c(T')$ for two trees that differ in a clean tree rotation.
Edges of the same color lead to the same subtrees.
The arrows indicate down-rotations of the vertex~$b$ by 1 step.}
\label{fig:child-rot}
\end{figure}

\subsection{Proof of Theorem~\ref{thm:GNk-ham}}

\begin{figure}
\makebox[0cm]{ % artificial box to center the picture
\includegraphics[page=7]{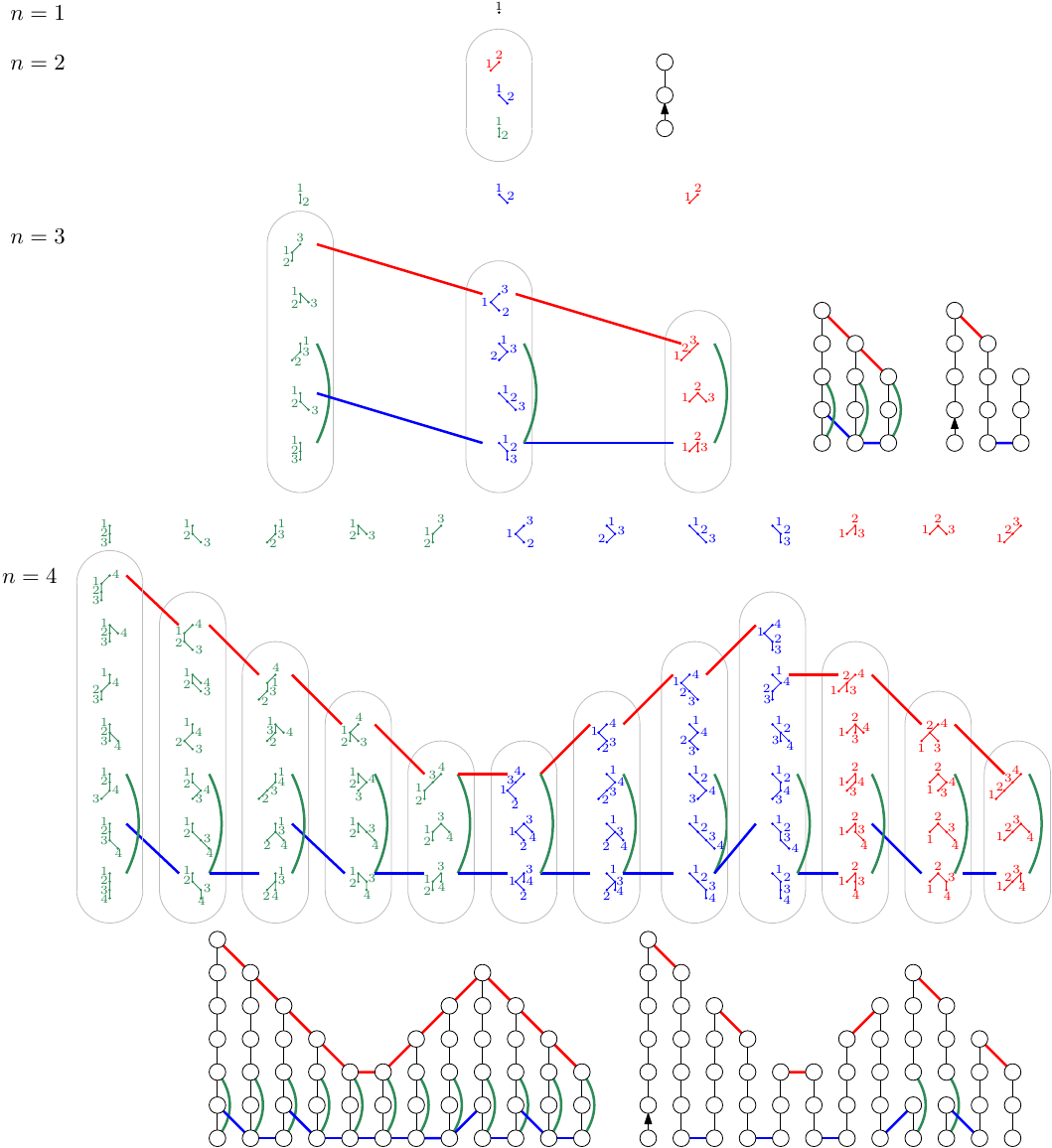}
}
\caption{Illustration of the proof of Theorem~\ref{thm:GNk-ham}.
For each tree~$T$ with $n=1,2,3$ vertices, the trees in the children sequence $c(T)$ are drawn in the bubble below it, to be read from top to bottom (in forward order).
The red lines indicate tree rotations among first entries of the children sequences corresponding to~\eqref{eq:frot}.
The blue lines indicate tree rotations among last entries of the children sequences, where the horizontal blue lines correspond to~\eqref{eq:lrot1} and the tilted blue lines correspond to~\eqref{eq:lrot2}.
The green lines indicate tree rotations among the last $k$ entries of the children sequence; recall~\eqref{eq:Cell}.
For comparison, the tree structure shown in Figure~5 of~\cite{MR2474724} has 3 groups of children of sizes 3,3,6 on level~$n=3$, whereas our tree has 3 groups of sizes 3,4,5.}
\label{fig:gc3}
\end{figure}

\begin{figure}
\makebox[0cm]{ % artificial box to center the picture
\includegraphics[page=8]{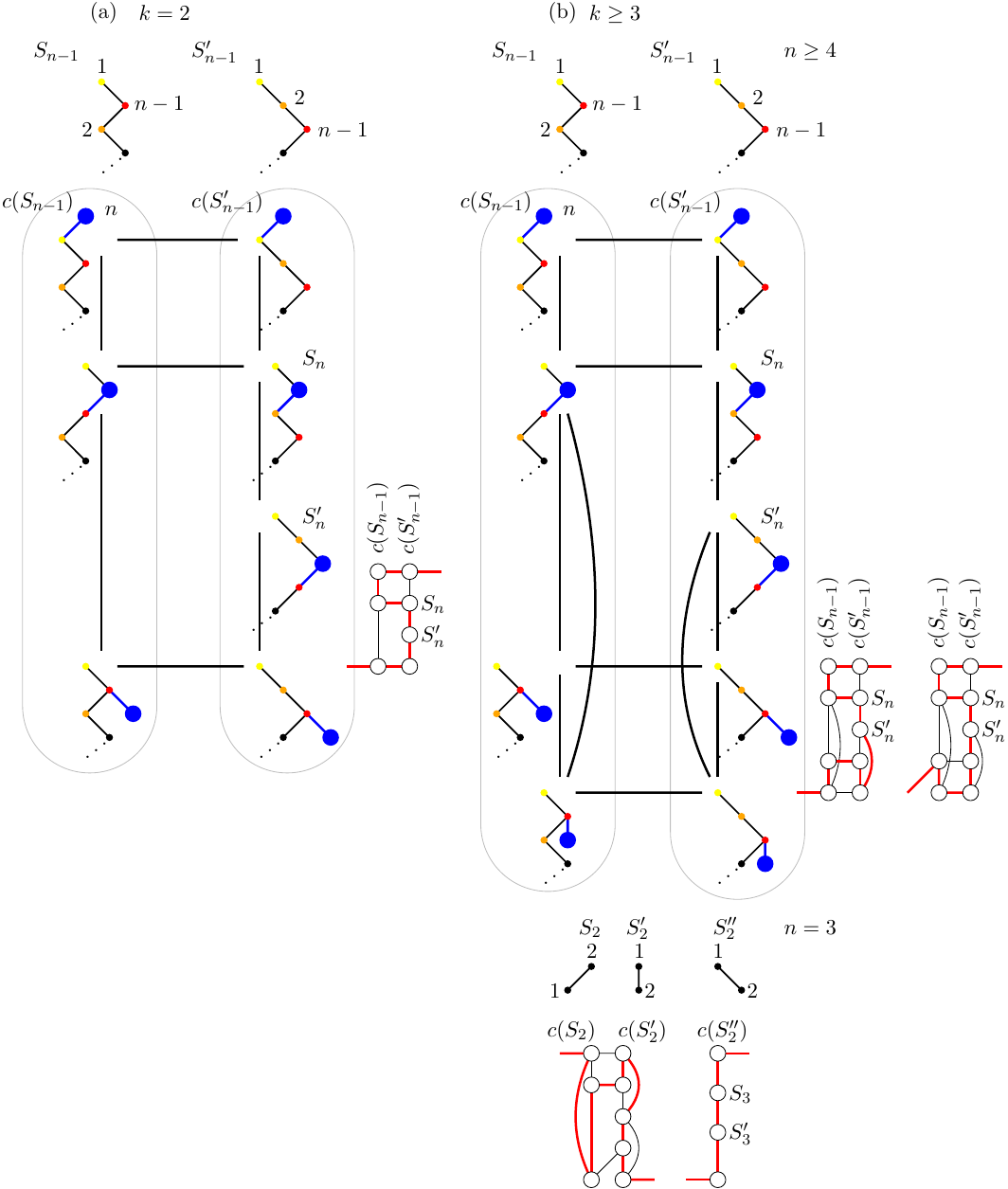}
}
\caption{Illustration of the construction of a Hamilton cycle in the proof of Theorem~\ref{thm:GNk-ham}.
Specifically, the figure shows the interleaving of the children sequences $c(S_{n-1})$ and~$c(S_{n-1}')$ of two special trees~$S_{n-1}$ and~$S_{n-1}'$.
As in Figure~\ref{fig:gc3}, the children sequences are displayed below each tree from top to bottom (corresponding to forward order of the sequence).
Part (a) shows the construction for $k=2$ (binary trees), and part (b) shows the construction for $k\geq 3$, subdivided into the two cases $n\geq 4$ (top right) and $n=3$ (bottom right).}
\label{fig:hc}
\end{figure}

\begin{proof}[Proof of Theorem~\ref{thm:GNk-ham}]
For the reader's convenience, the proof is illustrated in Figures~\ref{fig:gc3} and~\ref{fig:hc}.

To construct a Hamilton cycle in~$\cG_{N,k+1}$, we view it as the rotation graph of all $n$-vertex $k$-ary trees (recall that $N=(k-1)n+2$), and we argue by induction on~$n$.
The induction basis $n=\max\{2,5-k\}$ is trivial.
Indeed, for $k=2$ we have $n=3$ and $N=5$ and the graph $\cG_{5,2}$ is a 5-cycle, and for $k\geq 3$ we have $n=2$ and $N=2k$ and the graph $\cG_{2k,k+1}$ is a complete graph on $k$ vertices.
For the induction step, let $n>\max\{2,5-k\}$.
By induction, we have a cyclic ordering $C=(T_1,\ldots,T_L)$ of the set~$\cT_{n-1,k}$ of all $k$-ary trees with $n-1$ vertices, such that any two consecutive trees (including $T_L$ and $T_1$) differ in a rotation.
We replace in the sequence~$C$ each tree~$T_i$ by its children sequence~$c(T_i)$ if $i$ is even or the reversed children sequence $\rev(c(T_i))$ if $i$ is odd, yielding a sequence of all $k$-ary trees with $n$ vertices.
Within each subsequence~$c(T_i)$ or~$\rev(c(T_i))$, any two consecutive trees differ in a rotation as argued in Section~\ref{sec:cseq}.
Furthermore, on every transition $\rev(c(T_i))\rightarrow c(T_{i+1})$ the first entries of~$c(T_i)$ and~$c(T_{i+1})$ differ in a rotation by~\eqref{eq:frot}.
Lastly, on every transition $c(T_i)\rightarrow \rev(c(T_{i+1}))$, there are two possible cases:
Either the last entries of~$c(T_i)$ and~$c(T_{i+1})$ differ in a rotation by~\eqref{eq:lrot1}, or the last entry of one of the two sequences differs in the $r$th-last entry of the other by~\eqref{eq:lrot2} for some $r<k$, in which case we reverse the last $r$ elements of the other sequence, which is possible since any two among the last $k$ trees of each sequence differ in a rotation by~\eqref{eq:Cell} (i.e., these $k$ trees form a clique in the rotation graph); see Figure~\ref{fig:gc3}.
These modifications yield an ordering of the set~$\cT_{n,k}$ of all $k$-ary trees with $n$ vertices by rotations, i.e., a Hamilton path in the graph~$\cG_{N,k+1}$.
If $L=t_{n-1,k}=|\cT_{n-1,k}|$ is even, then the last transition $c(T_L)\rightarrow \rev(c(T_1))$ can be treated in the same way as before, yielding a Hamilton cycle.
On the other hand, if $L$ is odd, then we correct for this parity issue by interleaving the children sequences~$c(S_{n-1})$ and~$c(S_{n-1}')$ of two particular trees~$S_{n-1},S_{n-1}'\in\cT_{n-1,k}$ as shown in Figure~\ref{fig:hc}, maintaining the invariant that the Hamilton cycle visits~$S_n$ and~$S_n'$ consecutively.
\end{proof}

\section{Proofs of Theorems~\ref{thm:GNk-algo} and~\ref{thm:FN-algo}}
\label{sec:algo}

In the following proofs we sketch the basic data structures and ideas used in implementing our algorithms; for details see our C++ implementation~\cite{cos_kary}.

\begin{proof}[Proof of Theorem~\ref{thm:GNk-algo}]
We represent a $k$-ary tree with vertices $1,\ldots,n$ by maintaining an array of length~$k$ for the children of each vertex.
Furthermore, each vertex has access to its parent vertex.
With these data structures, up- and down-rotations by 1 step can be performed in time~$\cO(k)$.
The algorithm to generate all trees by rotations is a faithful implementation of the inductive zigzag construction described in Section~\ref{sec:unified}.
As we are only computing a Hamilton path in the rotation graph (not a cycle), the interleaving of children sequences for the special trees~$S_n$ and~$S_n'$ discussed in the proof of Theorem~\ref{thm:GNk-ham} is omitted, which simplifies this computation.
Specifically, the selection of the vertex to be rotated next and in which direction can be achieved in time~$\cO(1)$ using simple data structures developed for the Hartung-Hoang-M\"utze-Williams permutation language framework; see e.g.~\cite{MR4598046,DBLP:journals/talg/CardinalMM25}.

The only algorithmic detail that requires careful attention is the reversal of suffixes of the children sequences that occurs on some transitions $c(T)\rightarrow \rev(c(T'))$ between $n$-vertex trees, where $T\rightarrow T'$ is a transition on the Hamilton path in the rotation graph of trees with $n-1$ vertices.
Note that this phenomenon only materializes for $k\geq 3$, as for binary trees ($k=2$) we are always in case~\eqref{eq:lrot1} but never in case~\eqref{eq:lrot2}.
To achieve these suffix reversals, whenever a vertex starts a sequence of down-rotations or up-rotations, we eagerly precompute the entire sequence of rotations necessary to obtain the children sequence, without applying the rotations yet.
Specifically, we encode an up-rotation by $s$ steps by the integer~$+s$, and a down-rotation by $s$ steps by the integer~$-s$.
Consequently, a children sequence~$c(T)$ of length~$r$ corresponds to the \defi{step sequence} $(-1)^{r-1}$, whereas the reverse sequence $\rev(c(T))$ corresponds to the step sequence $(+1)^{r-1}$.
A reversal of a suffix of~$c(T)$ of length~$s$ corresponds to the modification $(-1)^{r-1}\rightarrow (-1)^{r-1-s},-s,(+1)^{s-1}$ of the step sequence.
These precomputations allow to predict the required suffix reversals.
Importantly, suffix reversals also result in rotations by $>1$ steps (whenever entries of the step sequence are different from $\pm 1$).
Observe however that the costs for precomputing the step sequence and the rotations of $>1$ steps are proportional to the length of the children sequence and are consequently amortized to~$\cO(k)$ on average per generated tree.
\end{proof}

\begin{proof}[Proof of Theorem~\ref{thm:FN-algo}]
We first consider the case that the number~$N$ of points is even, and then describe the modifications necessary to also cover the case that $N$ is odd.
We represent a colorful triangulation~$T\in\cC_N$ by a pair~$(r(T),b(T))$, where $r(T)$ is the quadrangulation with $q=(N-2)/2$ quadrangles obtained from~$T$ by removing all monochromatic edges, represented by its dual ternary tree (recall Section~\ref{sec:alternating}), and $b(T)\in\{0,1\}^q$ is a bitstring that describes for each quadrangle of~$r(T)$ whether the monochromatic edge of~$T$ inside of it connects two red points or two blue points (by setting the corresponding bit to~0 or~1, respectively).
Recall that the reduced flip graph~$\cF_N'$ on quadrangulations is isomorphic to the rotation graph of ternary trees, specifically $\cF_N'\simeq \cG_{N,4}$, i.e., we can use our algorithm for computing a Hamilton path (not cycle) in~$\cG_{N,4}$ discussed earlier.
As we are only computing a Hamilton path in~$\cF_N$ (not a cycle), we slightly modify the construction described in the proof of Lemma~\ref{lem:FN-ham-even}~(i), namely we replace each edge of the Hamilton path in~$\cF_N'$ by a single (connector) edge in~$\cF_N$, not by two edges, using the property that the hypercube is \defi{hamilton-laceable}, i.e., it admits a Hamilton path between any two end vertices of opposite parity.
Along the resulting Hamilton path in~$\cF_N$, we alternatingly see one flip of a colorful edge (=tree rotation) followed by $2^q-1$ flips of monochromatic edges (=flipping a bit inside a hypercube).
One can check that in our construction one only encounters parity differences~$1$ or~$3$, and such Hamilton paths in the hypercube can be computed efficiently, i.e., in time~$\cO(1)$ per vertex, by the binary reflected Gray code, or by straightforward modifications of it.
The tree rotations take time~$\cO(k)=\cO(1)$ (as the arity $k=3$ is fixed) on average, so overall the running time is~$\cO(1)$ on average per generated colorful triangulation.

It remains to discuss the case that the number~$N$ of points in the triangulations is odd, i.e., the points~$1$ and~$N$ have the same color (both red), whereas the color alternates between the remaining points.
The corresponding reduced graph~$\cF_N'$ has $q=(N-3)/2$ quadrangles and one triangle~$t$ with the monochromatic edge~$(1,N)$ on the boundary (recall Section~\ref{sec:general} and Figures~\ref{fig:F4567} and~\ref{fig:F89}).
We model this in our data structures by including an artificial point~$N+1$ connected to~1 and~$N$, turning the triangle~$t$ into an artificial quadrangle containing~$t$, subject to the constraint that the monochromatic edge~$(1,N)$ inside this quadrangle must never be flipped.
Hence, we can use the aforementioned data structures (ternary tree plus bitstring), with one extra tree vertex and one extra bit, where the extra tree vertex only ever appears on the rightmost branch (this corresponds to the artificial quadrangle always containing points~$1$ and~$N$), and the extra bit never changes (as the monochromatic edge~$(1,N)$ is never flipped).
Both constraints can be handled by straightforward adjustments.
\end{proof}

\section{Proof of Theorem~\ref{thm:HN-conn}}
\label{sec:3col}

In this section, we prove Theorem~\ref{thm:HN-conn}.
Recall that we consider triangulations for which the points $1,\ldots,N$ are colored red ($\tr$), blue ($\tb$) and green ($\tg$) alternatingly, and where every triangle has points of all three colors.
As discussed in Section~\ref{sec:prelim-3col}, the corresponding binary trees with $N-2$ vertices have the property that every vertex in odd distance from the root has two children.
A twist corresponds to subtree modifications of the form $T=[[[A,B],[C,D]],E]\longleftrightarrow [A,[[B,C],[D,E]]]=T'$.
Specifically, we refer to $T\rightarrow T'$ as a \defi{right-twist} of the subtree~$T$, and to the inverse operation $T\leftarrow T'$ as a \defi{left-twist} of the subtree~$T'$; see Figure~\ref{fig:twist}.
Note that for a right-twist to be applicable to~$T$, the left child of the root of~$T$ must have two children.
Similarly, for a left-twist to~$T'$ to be applicable, the right child of the root of~$T'$ must have two children.

\begin{proof}[Proof of Theorem~\ref{thm:HN-conn}]
Let $T_0$ be the binary tree with $N-2$ vertices for which every vertex at even distance from the root has only a left child, and every vertex at odd distance from the root has two children, where the right child has no children.
We show that in the flip graph~$\cH_N$, every node~$T$ is connected to~$T_0$ by a sequence of twists, applied to the binary tree representations.
This argument is split into two parts.

Let $v\in T$ be a vertex at even distance from the root.
Then there is a sequence of twists $T\rightarrow T'$ that modifies only the subtree rooted at~$v$ such that in~$T'$ the root of this subtree has no right child.
Indeed, if $v$ has a right child~$w$, then $w$ is at odd distance from the root, which implies that $w$ has two children, and so we can apply a left-twist in~$T$ to the subtree rooted at~$v$, which reduces the number of descendants in the right subtree of its root.
Consequently, we can repeat this operation until the root of this subtree has no right child.

Let $v\in T$ be a vertex at odd distance from the root, and let $u$ and~$w$ be its left and right child, respectively.
Then there is a sequence of twists $T\rightarrow T'$ that modifies only the subtree rooted at~$v$ such that in~$T'$ the right child of the root of this subtree has no children; see Figure~\ref{fig:conn}.
Let $x$ be the parent of~$v$.
We only consider the case that~$v$ is the left child of~$x$, as the other case is analogous.
Let $S=[[[A,B],[C,D]],E]$ be the subtree of~$T$ rooted at~$x$, i.e., $[A,B]$ is the subtree rooted at~$u$ and~$[C,D]$ is the subtree rooted at~$w$.
As $w$ has even distance from the root, by our earlier arguments there is a sequence of twists $T\rightarrow T_1$ such that the subtree~$S_1$ rooted at~$x$ in~$T_1$ equals~$S_1=[[[A,B],[C',D']],E]$ with $D'=\varepsilon$, i.e., in~$T_1$ the right child of~$v$ has no right child.
We apply a single right-twist to the subtree~$S_1$ to obtain a tree~$T_2$ in which the corresponding subtree~$S_2$ equals $S_2=[A,[[B,C'],[D',E]]]$ (this temporarily modifies a subtree outside of the original subtree of~$T$ rooted at~$v$, but this is change is undone subsequently).
The subtree~$[B,C']$ of~$T_2$ has its root~$v$ at even distance from the root, so again there is a sequence of twists~$T_2\rightarrow T_3$ such that the corresponding subtree~$S_3$ in~$T_3$ becomes~$S_3=[A,[[B',C''],[D',E]]]$ with~$C''=\varepsilon$.
Applying a left-twist to the subtree~$S_3$ yields a tree~$T'$ in which the corresponding subtree (rooted at~$x$) is now~$S'=[[[A,B'],[C'',D']],E]$ with $C''=D'=\varepsilon$, as desired.

\begin{figure}
\makebox[0cm]{ % artificial box to center the picture
\includegraphics[page=2]{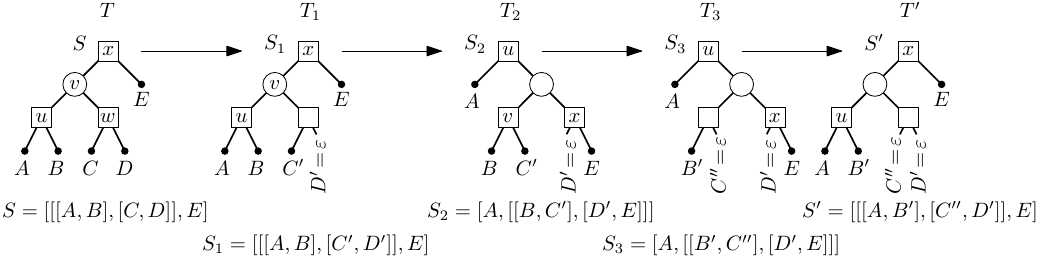}
}
\caption{Illustration of the proof of Theorem~\ref{thm:HN-conn}.
Vertices in even and odd distance from the root are drawn as squares and circles, respectively.}
\label{fig:conn}
\end{figure}

Clearly, from any tree~$T$ we can alternatingly apply the sequences of operations discussed before, starting at the root and then descending towards the left subtree so that~$T$ is successively transformed into~$T_0$.
This completes the proof.
\end{proof}

\section{Open questions}
\label{sec:open}

Tables~\ref{tab:hc} and~\ref{tab:hp} show all coloring sequences~$\alpha$ on up to $N\leq 11$ points for which the graph~$\cF_\alpha$ has no Hamilton path or cycle.
The sequences are shown up to rotational symmetry, reversal, and exchange of the two colors.
In several cases, Theorems~\ref{thm:comb}, \ref{thm:grid} and~\ref{thm:a11111} provide an explanation for the non-Hamiltonicity; see the third column in the tables.
In the other cases, we are still missing such an explanation.

\begin{table}
\begin{minipage}[t]{.5\linewidth}
\centering
\caption{Coloring sequences~$\alpha$ for $N\leq 11$ for which~$\cF_\alpha$ has no Hamilton cycle but a Hamilton path.}
\label{tab:hc}
\begin{tabular}{lll}
$N$ & $\alpha$ & Reason \\\hline
6 & $(4,2)$ & Thm.~\ref{thm:comb} \\
  & $(2,1,1,2)$ & \\\hline
7 & $(5,2)$ & Thm.~\ref{thm:comb} \\
  & $(2,1,1,1,1,1)$ & Thm.~\ref{thm:a11111} \\\hline
8 & $(6,2)$ & Thm.~\ref{thm:comb} \\
  & $(4,4)$ & Thm.~\ref{thm:comb} \\\hline
9 & $(7,2)$ & Thm.~\ref{thm:comb} \\
  & $(4,1,1,1,1,1)$ & Thm.~\ref{thm:a11111} \\\hline
10& $(8,2)$ & Thm.~\ref{thm:comb} \\
  & $(6,4)$ & Thm.~\ref{thm:comb} \\
  & $(5,1,2,2)$ & \\
  & $(5,1,1,1,1,1)$ & Thm.~\ref{thm:a11111} \\
  & $(4,1,3,2)$ & \\
  & $(4,1,2,3)$ & \\\hline
11& $(9,2)$ & Thm.~\ref{thm:comb} \\
  & $(6,1,1,1,1,1)$ & Thm.~\ref{thm:a11111} \\
  & $(5,1,2,3)$ & \\
  & $(5,1,1,4)$ & \\
\end{tabular}
\end{minipage}%
\begin{minipage}[t]{.5\linewidth}
\centering
\caption{Coloring sequences~$\alpha$ for $N\leq 11$ for which~$\cF_\alpha$ has no Hamilton path.}
\label{tab:hp}
\begin{tabular}{lll}
$N$ & $\alpha$ & Reason \\\hline
6 & $(3,1,1,1)$ & Thm.~\ref{thm:grid} \\
  & $(3,3)$ & Thm.~\ref{thm:comb} \\
  & $(1,1,1,1,1,1)$ & Thm.~\ref{thm:a11111} \\ \hline
7 & $(4,3)$ & Thm.~\ref{thm:comb} \\
  & $(3,1,2,1)$ & Thm.~\ref{thm:grid} \\
  & $(3,1,1,2)$ & \\\hline
8 & $(5,1,1,1)$ & Thm.~\ref{thm:grid} \\
  & $(5,3)$ & Thm.~\ref{thm:comb} \\
  & $(4,1,2,1)$ & Thm.~\ref{thm:grid} \\
  & $(4,1,1,2)$ & \\
  & $(3,1,3,1)$ & Thm.~\ref{thm:grid} \\
  & $(3,1,1,1,1,1)$ & Thm.~\ref{thm:a11111} \\
  & $(3,1,1,3)$ & \\
  & $(3,2,1,2)$ & \\\hline
9 & $(6,1,1,1)$ & Thm.~\ref{thm:grid} \\
  & $(6,3)$ & Thm.~\ref{thm:comb} \\
  & $(5,1,2,1)$ & Thm.~\ref{thm:grid} \\
  & $(5,1,1,2)$ & \\
  & $(5,4)$ & Thm.~\ref{thm:comb} \\
  & $(4,1,3,1)$ & Thm.~\ref{thm:grid} \\
  & $(4,1,1,3)$ & \\
  & $(3,1,3,2)$ & \\
  & $(3,1,2,3)$ & \\\hline
10& $(7,1,1,1)$ & Thm.~\ref{thm:grid} \\
  & $(7,3)$ & Thm.~\ref{thm:comb} \\
  & $(6,1,2,1)$ & Thm.~\ref{thm:grid} \\
  & $(6,1,1,2)$ & \\
  & $(5,1,3,1)$ & Thm.~\ref{thm:grid} \\
  & $(5,1,1,3)$ & \\
  & $(5,5)$ & Thm.~\ref{thm:comb} \\
  & $(4,1,4,1)$ & Thm.~\ref{thm:grid} \\
  & $(4,1,1,4)$ & \\
  & $(3,1,3,3)$ & \\\hline
11& $(8,1,1,1)$ & Thm.~\ref{thm:grid} \\
  & $(8,3)$ & Thm.~\ref{thm:comb} \\
  & $(7,1,2,1)$ & Thm.~\ref{thm:grid} \\
  & $(7,1,1,2)$ & \\
  & $(7,4)$ & Thm.~\ref{thm:comb} \\
  & $(6,1,3,1)$ & Thm.~\ref{thm:grid} \\
  & $(6,1,1,3)$ & \\
  & $(6,5)$ & Thm.~\ref{thm:comb} \\
  & $(5,1,4,1)$ & Thm.~\ref{thm:grid} \\
  & $(5,1,3,2)$ & \\
  & $(4,1,3,3)$ & \\
\end{tabular}
\end{minipage}
\end{table}

Based on this data, we feel that Theorem~\ref{thm:Falpha-ham} can be strengthened and the requirement $\ell\geq 10$ relaxed to $\ell\geq 8$.
Furthermore, it seems that for $\ell=6$ there is always a Hamilton path in~$\cF_\alpha$ unless $\alpha\in\{(1,1,1,1,1,1),(3,1,1,1,1,1)\}$.

Also, Theorem~\ref{thm:HN-conn} on twists in 3-colored triangulations invites deeper investigation.
We conjecture that the graph~$\cH_N$ has a Hamilton cycle for all $N\geq 9$ that are divisible by~3.
Furthermore, it seems that if $N=2\pmod{3}$ the graph~$\cH_N$ is not connected.
What are the properties of the flip graphs resulting from general coloring patterns with three or more colors?

Another interesting question concerns bijections between $k$-ary trees and classes of permutations.
For $k=2$ (binary trees), there is a natural bijection to 231-avoiding permutations.
Are there similar correspondences between $k$-ary trees and pattern-avoiding permutations for $k\geq 3$?
In particular, do tree rotations translate to nice operations on the permutations, specifically to so-called jumps heavily used in~\cite{MR4391718,MR4344032,MR4598046,DBLP:journals/talg/CardinalMM25}?

Going back to the uncolored setting and the associahedron~$\cG_N$, Theorem~\ref{thm:Falpha-ham} shows that~$\cG_N$ admits cycles of many different lengths.
What is the \defi{cycle spectrum} of~$\cG_N$, i.e., the set~$S(\cG_N)$ of all possible lengths of cycles in~$\cG_N$? We conjecture that almost all lengths are possible.

\begin{conjecture}
We have $|S(\cG_N)|/|\cD_{N,3}|=1-o(1)$ as $N\rightarrow\infty$.
\end{conjecture}

Baur, Bergerova, Voon and Xu~\cite{baur-et-al:24} recently introduced another family of flip graphs on triangulations in which the triangles are colored, not the vertices.
The resulting graphs are disconnected in general, and their structure is still not very well understood  (in~\cite{MR1932681} these graphs are related to the famous Four Color Theorem).

\bibliographystyle{alpha}
\bibliography{refs}

\end{document}